\documentclass[11pt,a4paper]{article}
\usepackage{amsmath}
\usepackage{amsthm}
\usepackage{cases}
\usepackage{tikz}
\usepackage{amsfonts}
\usepackage{multirow}
 \usepackage{tikz-cd}
\usepackage{subcaption}
\usetikzlibrary{arrows,shapes}
\usetikzlibrary{matrix}
\usepackage{float}
\usepackage{algorithm}
\usepackage{algpseudocode}
\usepackage{hyperref}
\usepackage{tabularx}
\usepackage{amssymb}
\usepackage{graphicx}
\usepackage{tcolorbox}
\usepackage{booktabs}
\usepackage{rotating,tabularx}

\usepackage[left=2cm,right=2cm,top=2cm,bottom=2cm]{geometry}
\newcolumntype{L}[1]{>{\raggedright\let\newline\\\arraybackslash\hspace{0pt}}m{#1}}
\newcolumntype{C}[1]{>{\centering\let\newline\\\arraybackslash\hspace{0pt}}m{#1}}
\newcolumntype{R}[1]{>{\raggedleft\let\newline\\\arraybackslash\hspace{0pt}}m{#1}}

\newtheorem{Theorem}{Theorem}[section]
\newtheorem{Proposition}[Theorem]{Proposition}
\newtheorem{Remark}[Theorem]{Remark}
\newtheorem{Lemma}[Theorem]{Lemma}
\newtheorem{Corollary}[Theorem]{Corollary}
\newtheorem{Definition}[Theorem]{Definition}
\newtheorem{Example}[Theorem]{Example}

\usepackage{hyperref}
\hypersetup{colorlinks=true,urlcolor=blue}
\expandafter\let\expandafter\oldproof\csname\string\proof\endcsname
\let\oldendproof\endproof
\renewenvironment{proof}[1][\proofname]{
\oldproof[\ttfamily\scshape \bf #1.]
}{\oldendproof}
\def\ve{\varepsilon}
\def\tilde{\widetilde}
\def\emp{\emptyset}

\def\dist{{\rm dist}}
\def\dom{{\rm dom}\,}

\def\min{\mbox{\rm minimize}}

\def\B{\mathbb B}

\def\ox{\overline{x}}
\def\oy{\overline{y}}
\def\oz{\overline{z}}

\def\disp{\displaystyle}
\def\tto{\rightrightarrows}
\def\Hat{\widehat}

\def\Bar{\overline}
\def\ra{\rangle}
\def\la{\langle}
\def\ve{\varepsilon}
\def\epsilon{\varepsilon}
\def\ox{\bar{x}}
\def\oy{\bar{y}}
\def\oz{\bar{z}}

\def\ov{\bar{v}}

\def\gph{\mbox{\rm gph}\,}

\def\dom{\mbox{\rm dom}\,}
\def\ker{\mbox{\rm ker}\,}

\def\dn{\downarrow}
\def\O{\Omega}
\def\ph{\varphi}
\def\emp{\emptyset}

\def\oR{\Bar{\R}}
\def\lm{\lambda}

\def\al{\alpha}
\def\Th{\Theta}
\def \N{{\rm I\!N}}
\def \R{{\rm I\!R}}

\def\Limsup{\mathop{{\rm Lim}\,{\rm sup}}}

\def\Limsup{\mathop{{\rm Lim}\,{\rm sup}}}

\numberwithin{equation}{section}

\begin{document}

\title{\bf Second-Order Subdifferential Optimality Conditions\\ in  Nonsmooth Optimization}
\author{Pham Duy Khanh\footnote{Department of Mathematics, Ho Chi Minh City University of Education, Ho Chi Minh City, Vietnam. E-mail: khanhpd@hcmue.edu.vn.} \quad Vu Vinh Huy Khoa\footnote{Department of Mathematics, Wayne State University, Detroit, Michigan, USA. E-mail: khoavu@wayne.edu. Research of this author was partly supported by the US National Science Foundation under grant DMS-2204519.}\quad Boris S. Mordukhovich\footnote{Department of Mathematics, Wayne State University, Detroit, Michigan, USA. E-mail: aa1086@wayne.edu. Research of this author was partly supported by the US National Science Foundation under grant DMS-2204519, by the Australian Research Council under Discovery Project DP-190100555, and by Project~111 of China under grant D21024.}\quad Vo Thanh Phat\footnote{Department of Mathematics and Statistics,  University of North Dakota, Grand Forks, North Dakota, USA. E-mail: thanh.vo.1@und.edu. Research of this author was partly supported by the US National Science Foundation under grant DMS-2204519.}}
\maketitle
\vspace*{-0.2in}
\begin{quote}
{\small\noindent {\bf Abstract.} The paper is devoted to deriving novel second-order necessary and sufficient optimality conditions for local minimizers in rather general classes of nonsmooth unconstrained and constrained optimization problems in finite-dimensional spaces. The established  conditions are expressed in terms of second-order subdifferentials of lower semicontinuous functions and mainly concern prox-regular objectives that cover a large territory in nonsmooth optimization and its applications. Our tools are based on the machinery of variational analysis and second-order generalized differentiation. The obtained general results are applied to problems of nonlinear programming, where the derived second-order optimality conditions are new even for problems with twice continuously differential data, being expressed there in terms of the classical Hessian matrices.\\\vspace*{0.03in}\noindent
\vspace*{0.03in}\noindent
}
\end{quote}

\section{Introduction}\label{intro}\vspace*{-0.05in}

It has been well recognized in optimization and variational analysis that second-order optimality conditions (both necessary and sufficient) play a highly important role in qualitative aspects of optimization as well as in the design and justification of numerical algorithms. The derivation of such conditions in problems of {\em second-order nonsmooth optimization} (by which we understand those where at least one of the involved functions is not twice continuously differentiable) naturally requires appropriate notions of second-order generalized differentiation and the machinery of  second-order variational analysis.

In this paper, we explore for these purposes {\em second-order subdifferentials} (or {\em generalized Hessians}) of lower semicontinuous (l.s.c.) functions initiated by Mordukhovich \cite{Mor92} in the general extended-real-valued framework. These constructions have been largely studied and utilized in numerous aspects of variational analysis, variational stability, optimization, control, and various applications; see, e.g., the books \cite{Mordukhovich06,Mor18,Mor24} and the references therein, where the reader can find, in particular, comprehensive calculus rules and explicit computations of second-order subdifferentials and related constructions for broad classes of extended-real-valued functions that overwhelmingly arise in optimization.

A prominent role of second-order subdifferentials was revealed in the study of {\em tilt-stable local minimizers}, the notion introduced by Poliquin and Rockafellar \cite{Poli} in the framework of unconstrained extended-real-valued optimization, where the basic second-order subdifferential of \cite{Mor92} was employed to establish {\em pointbased} (i.e., formulated {\em at} the point in question) {\em characterizations} of tilt-stable minimizers for the large class of prox-regular and subdifferentially continuous functions. Using second-order subdifferential calculus rules and other techniques of second-order variational analysis, both pointbased and {\em neighborhood} (i.e., involving points {\em nearby} the reference one)  characterizations of tilt-stable local minimizers were established for various classes of problems in constrained optimization; see \cite{bgm,dima,gm15,lewis,Mor24,mnghia-tilt13,MorduNghia,MorOu13,MorduHector,mor-roc,jafar,zheng} among many publications in this direction. Furthermore, the obtained characterizations of tilt stability laid the foundation for the design and justification of generalized Newtonian methods to solve nonsmooth optimization and related problems by using second-order subdifferentials; see the recent developments in  \cite{BorisKhanhPhat,kmptjogo,kmptmp,Mor24,BorisEbrahim}.

Besides applications to tilt-stable minimizers, second-order subdifferentials have been instrumental to study arbitrary local minimizers in some special classes of nonsmooth optimization problems. Using a symmetric version of the basic second-order subdifferential, Huy and Tuyen \cite{huytuyen16} derived pointbased no-gap necessary and sufficient conditions for local solutions to unconstrained problems of minimizing objectives of class ${\cal C}^{1,1}$, i.e., continuously differentiable functions with locally Lipschitzian gradients. Neighborhood conditions in terms of some modifications of the basic second-order subdifferential were obtained by Chieu et al.\ \cite{ChieuLeeYen17} for unconstrained problems with ${\cal C}^{1,1}$ objectives and also for merely continuously differentiable (${\cal C}^1$-smooth) functions under additional assumptions. Some extensions of these results to optimization problems with specific types of constraints were given in \cite{anyenxu23,anyen21}.
\vspace*{0.03in}

The aim of this paper is to derive new second-order subdifferential necessary and sufficient conditions for local minimizers and its strong counterparts in various formats of nonsmooth unconstrained and constrained optimization. First we consider minimization problems written in the {\em unconstrained extended-real-valued} form. In contrast to the previous developments in \cite{ChieuLeeYen17,huytuyen16} for intrinsically unconstrained problems with smooth data, we now address problems with extended-real-valued {\em prox-regular} objectives, which cover a much broader class of optimization problems and allow us to implicitly incorporate {\em geometric constraints} as in \cite{anyenxu23,anyen21} via the domain sets of the corresponding objectives. For the general class of prox-regular functions without any additional assumptions, we establish a {\em second-order necessary optimality condition} expressed in the {\em pointbased} form via the so-called combined version of the basic second-order subdifferential; see below. To derive next its {\em no-gap neighborhood} counterpart as a {\em second-order sufficient optimality condition}, we first obtain a new {\em characterization} of {\em variational convexity} of extended-real-valued functions, the property of its own great importance as documented in \cite{r19} and \cite{KKMP23-2,kmp22convex} which yields the sufficiency for local minimizers. Then we consider nonsmooth optimization problems with {\em explicit functional-geometric constraints} including, in particular, problems of {\em conic programming}. The usage of second-order subdifferential calculus rules allows us to establish no-gap second-order necessary and sufficient optimality conditions under fairly mild constraint qualifications. Finally, the paper develops constructive applications of the obtained second-order subdifferential conditions to problems of {\em nonlinear programming} (NLP), where we arrive at new results even in the classical NLP framework with twice continuously differentiable (${\cal C}^2$-smooth) data.\vspace*{0.03in}
 
The remainder of the paper is structured as follows. Section~\ref{sec:prelim} provides an overview of some key concepts in variational analysis and generalized differentiation, which are widely employed in the formulations and proofs of our major results. In Section~\ref{sec:moreau}, we establish the equivalence between local minimizers of a given l.s.c.\ function and local minimizers of its Moreau envelope. The main result of Section~\ref{sec:optiforprox} presents a novel necessary second-order subdifferential condition for local minimizers of extended-real-valued prox-regular functions. Section~\ref{sec:vc} provides a coderivative characterization of variational convexity for l.s.c.\ functions and applies it to deriving a second-order subdifferential sufficient optimality condition. In  Section~\ref{sec:strong}, we obtain counterparts of the above second-order optimality conditions for the case of strong minimizers. Sections~\ref{sec:optiforconstrained} and \ref{sec:2ndsuf} are devoted, respectively, to establishing second-order subdifferential optimality conditions for local and strong local minimizers of problems with functional-geometric constraints. Their further specifications and applications to nonlinear programming are given in Section~\ref{sec:appnonl}. Finally, Section~\ref{sec:conclusion} serves as a conclusion with summarizing the major findings of the paper and listing significant topics of our future research.\vspace*{-0.1in}

\section{Preliminaries from Variational Analysis}\label{sec:prelim}\vspace*{-0.05in}

In this section, we review some background information from variational analysis and first-order generalized differentiation in finite-dimensional spaces used in what follows. Proofs and related material can be found in the books \cite{Mor18,Rockafellar98} and the references therein. Recall that $B_\ve(\ox)$  and $\B_\ve(\ox)$ stand for the open and closed balls around $\ox$ with radius $\ve>0$, respectively, and that $\N:=\{1,2,\ldots\}$.

Let $F\colon\R^n\tto\R^m$ be a set-valued mapping/multifunction between finite-dimensional Euclidean spaces with its {\em graph} and {\em kernel} defined, respectively, by
$$
\text{gph}\,F:=\big\{(x,y)\in\R^n\times\R^m\;\big|\;y\in F(x)\big\}\;\mbox{ and }\;\ker F:=\big\{x\in\R^n\;\big|\;0\in F(x)\big\}.
$$
The {\em inverse} mapping $F^{-1}\colon\R^m\tto\R^n$ of $F$ is  $F^{-1}(y):=\{x\in\R^n\;|\;y\in F(x)\big\}$, and the
(Painlev\'e-Kuratowski) \textit{outer limit} of $F$ as $x\rightarrow\bar{x}$ is given by
\begin{equation}\label{outer}
\underset{x\rightarrow\bar{x}}{\Limsup}\;F(x):
=\big\{y\in \R ^n\;\big|\,\exists \mbox{ sequences } x_k\to \bar x,\ y_k\rightarrow y\;\mbox {with }\;y_k\in F(x_k),\;k\in\N\big
\}.
\end{equation}

Let $\varphi:\R ^n\rightarrow\overline{\R }:=(-\infty,\infty]$ be a proper extended-real-valued function with the domain $\dom\ph:=\{x\in\R^n\;|\;\ph(x)<\infty\}\ne\emp$. For $x\in\dom\ph$, define the (Fr\'echet) {\em regular subdifferential} of $\varphi$ at $x$ by
\begin{equation}\label{FrechetSubdifferential}
\Hat\partial\ph(x):=\Big\{v\in\R^n\;\Big|\liminf_{u\to x}\;\frac{\ph(u)-\ph(x)-\la v,u-x\ra}{\|u-x\|}\ge 0\Big\}
\end{equation}
with $\Hat\partial\ph(x):=\emp$ at $x\notin\dom\ph$, and the (Mordukhovich) {\em limiting subdifferential} of $\varphi$ at $\bar{x}\in\dom\ph$ by
\begin{equation}\label{MordukhovichSubdifferential}
\partial \varphi(\bar{x}):= \underset{x \overset{\varphi}{\to} \bar{x}} {\Limsup} \; \widehat{\partial}\varphi(x),
\end{equation}
where $x \overset{\varphi}{\to} \bar{x}$ means that $x\rightarrow\bar{x}$ with $\varphi(x)\rightarrow \varphi(\bar{x})$. Both regular and limiting subdifferentials of $\ph$ reduce to the classical gradient $\nabla\ph(\ox)$ of ${\cal C}^1$-smooth functions and to the classical subdifferential of convex analysis if $\ph$ is convex.

Given further a set $\Omega\subset\R ^n$ with its indicator function $\delta_\Omega(x)$ equal to $0$ for $x\in\Omega$ and to $\infty$ otherwise, the regular and the limiting \textit{normal cones} to $\Omega$ at $\bar{x}\in\Omega$ are defined via the corresponding subdifferentials \eqref{FrechetSubdifferential} and \eqref{MordukhovichSubdifferential} by, respectively,
\begin{equation}\label{NormalCones}
\widehat{N}_\Omega(\bar{x}):=\widehat{\partial}\delta_\Omega(\bar{x})
\quad\text{and}\quad
N_\Omega(\bar{x}):=\partial\delta_\Omega(\bar{x}).
\end{equation}
The	(Bouligand-Severi) {\em tangent/contingent cone} to $\Omega$ at $\ox$ is
\begin{equation}\label{tan}
T_\Omega(\ox):=\big\{w\in\R^n\;\big|\;\exists\,t_k\dn 0,\;w_k\to w\;\mbox{ as }\;k\to\infty\;\mbox{ with }\;\ox+t_k w_k\in\Omega\big\}.
\end{equation}
Note the duality correspondence $\Hat N_\O(\ox)=T_\O(\ox)^*$ between the regular normal and tangent cones. Observe to this end that the limiting normal cone in \eqref{NormalCones} cannot be tangentially generated due to its intrinsic nonconvexity.
The \textit{critical cone} to $\Omega$ at $(\ox,\ov)\in \gph {N}_\Omega$ is defined by
\begin{equation}\label{criticalconeK}
K_\Omega(\ox,\ov):=T_\Omega(\ox)\cap \{\ov\}^\perp.
\end{equation}

For a set-valued mapping $F\colon\R^n\tto\R^m$, the regular and limiting {\em coderivatives} of $F$ at $(\bar{x},\bar{y})\in\text{gph}\,F$ are defined via corresponding normal cones \eqref{NormalCones} by, respectively,
\begin{equation}\label{rcod}
\widehat D^*F(\bar{x},\bar{y})(u): =\big\{ v\in {\R ^n}\;\big|\;({v}, - {u}) \in \widehat{N}_{\text{gph}F}(\bar{x},\bar{y})\big\},
\end{equation}
\begin{equation}\label{lcod}
D^*F(\bar{x},\bar{y})(u) : =\big\{ v \in {\R ^n}\;\big|\;({v}, - {u}) \in {N}_{\text{gph}F}\,(\bar{x},\bar{y})\big\}
\end{equation}
for all $u\in\R^m$. We omit $\bar{y}$ in \eqref{rcod} and \eqref{lcod} if $F$ is single-valued. Let us mention that, despite their nonconvexity, the limiting constructions in \eqref{MordukhovichSubdifferential}, \eqref{NormalCones}, and \eqref{lcod} enjoy comprehensive {\em calculus rules} based on {\em variational/extremal principles} of variational analysis. On the other hand, the regular constructions in \eqref{FrechetSubdifferential}, \eqref{NormalCones}, and \eqref{rcod}, which are usually easier to compute, can be viewed by \eqref{outer} as efficient approximation tools to deal with the limiting ones.\vspace*{0.03in}

Next we formulate the {\em metric regularity} and {\em subregularity} properties of multifunctions that are highly recognized in variational
analysis and optimization. These properties are	frequently used in what follows. To proceed, recall that the distance function associated with a set $\Omega\subset\R^n$ is 
\begin{equation*} 
{\rm dist}(x;\Omega):=\inf\big\{\|w-x\|\;\big|\;w\in\Omega\big\},\quad
x\in\R^n. 
\end{equation*} 
 
\begin{Definition}\label{met-reg} {\rm Let $F\colon\R^n\tto\R^m$ be a set-valued mapping, and let $(\bar{x},\bar{y})\in\gph F$. We say that: 

{\bf(i)} $F$ is {\em metrically regular} around $(\bar{x},\bar{y})$ with modulus $\mu>0$ if there exist neighborhoods $U$ of $\bar{x}$
and $V$ of $\bar{y}$ ensuring that
\begin{equation*} 
{\rm dist}\big(x;F^{-1}(y)\big)\le\mu\,{\rm dist}\big(y;F(x)\big)\;\text{ for all }\;(x,y)\in U\times V. 
\end{equation*}
	 
{\bf(ii)} $F$ is {\em metrically subregular} at $(\bar{x},\bar{y})$
with modulus $\mu>0$ if there exists a neighborhood $U$ of $\bar{x}$ such that we have the estimate
\begin{equation*} 
{\rm dist}\big(x;F^{-1}(\bar{y})\big)\le\mu\,{\rm	dist}\big(\bar{y};F(x)\big)\;\text{ for all }\;x\in U.
\end{equation*}} 
\end{Definition}

A serious advantage of the limiting coderivative \eqref{lcod} is that it provides the following complete pointbased characterization of the metric regularity of any closed-graph multifunction $F\colon\R^n\tto\R^m$ around $(\ox,\oy)$ known as the {\em Mordukhovich coderivative criterion} \cite{Mordu93,Rockafellar98}:
\begin{equation}\label{cor-cr}
\ker D^*F(\ox,\oy)=\{0\}.
\end{equation}

The following notions were introduced by Poliquin and Rockafellar \cite{Poliquin}.

\begin{Definition}\label{prox-regfunction} \rm  A function $\varphi: \R^n \to \overline{\R}$ is said to be {\em prox-regular} at $\bar{x}$ for $\bar{v}\in\partial\ph(\ox)$ if $\varphi$ is finite and locally l.s.c.\ around $\bar{x}$, and if there exist numbers $\epsilon > 0$ and $r \geq  0$ such that 
\begin{equation}\label{prox}
\varphi(x) \geq \varphi(u) +\langle v, x- u \rangle  - \frac{r}{2}\|x-u\|^2 \quad  \text{for all} \quad x \in B_\epsilon(\bar{x})\;\mbox{ and }\;v\in \partial\varphi(x)
\end{equation} 
with $\|v-\bar{v}\| <\epsilon$, $\|u-\bar{x}\|< \epsilon$, and $\varphi(u)<\varphi(\bar{x})+\epsilon$. 
The function $\varphi$ is said to be {\em subdifferentially continuous} at $\bar{x}$ for $\bar{v}$ if $\bar{v} \in \partial\varphi(\bar{x})$ and, whenever $(x_k,v_k)\to(\bar{x},\ov)$ as $k\to\infty$ and $v_k \in \partial \varphi(x_k)$ for all $k\in\N$, we have $\varphi(x_k) \to \varphi(\bar{x})$ as $k\to\infty$.
\end{Definition}

If $\varphi$ is subdifferentially continuous at $\bar{x}$ for $\bar{v}$, the inequality  ``$\varphi(u)<\varphi(\bar{x})+\epsilon$" in the definition of prox-regularity above can be omitted. Functions that are both prox-regular and subdifferentially continuous are called \textit{continuously prox-regular}. As discussed in \cite{Rockafellar98}, the class of continuously prox-regular functions is fairly broad containing, besides $\mathcal{C}^2$-smooth and l.s.c.\ convex functions, also functions of class $\mathcal{C}^{1,1}$, lower-$\mathcal{C}^2$ functions, strongly amenable functions, etc. It has been recognized that continuously prox-regular functions play a central role in second-order variational analysis and its	applications; see the books \cite{Rockafellar98} and \cite{Mor24} with the commentaries and references therein. Nevertheless, our major results obtained in this paper show that the subdifferential continuity assumption can be generally dropped in the second-order necessary and sufficient conditions for local and strong local minimizers.\vspace*{-0.1in}

\section{Local Minimizers via Moreau Envelopes}\label{sec:moreau}\vspace*{-0.05in}

In this section, we verify the preservation of local minimizers of l.s.c.\ functions under taking the Moreau envelope associated with the function in question. Besides being important for their own sake, such a preservation occurs to be useful in our subsequent second-order derivations. 

For a proper l.s.c.\ function $\varphi: \R^n\to\overline{\R}$ and a parameter value $\lambda > 0$, the \textit{Moreau envelope} $e_\lambda\varphi$ and \textit{proximal mapping} $\text{\rm Prox}_{\lambda\varphi}$ are defined for all $x\in\R^n$ as follows:
\begin{equation}\label{Moreau} 
e_\lambda\varphi(x):= \inf\Big\{\varphi(y)+ \frac{1}{2\lambda}\|y-x\|^2\;\Big|\; y \in \R^n\Big\},
\end{equation}	
\begin{equation}\label{ProxMapping} 
\operatorname{Prox}_{\lambda\varphi}(x):= \operatorname{argmin} \Big\{\varphi(y)+ \frac{1}{2\lambda}\|y-x\|^2\;\Big|\; y \in \R^n\Big\}.
\end{equation}	
The function $\varphi$ is called \textit{prox-bounded} if there exists $\lambda>0$ such that $e_\lambda\varphi(x)>-\infty$ for some $x\in\R^n$. Equivalent descriptions of prox-boundedness are given in \cite[Exercise~1.24]{Rockafellar98}. To present now the important properties of Moreau envelopes and proximal mappings, which are taken from \cite[Proposition~13.37]{Rockafellar98}, we first recall that the {\em $\varphi$-attentive $\epsilon$-localization} of the subgradient mapping $\partial\varphi$ around $(\ox,\ov)\in\gph\partial\ph$ is the set-valued mapping $T\colon\R^n\tto\R^n$ defined by
\begin{equation}\label{attentive}
T(x):=\begin{cases}
\big\{v\in\partial\varphi(x)\;\big|\;\|v-\ov\|<\epsilon\big\} & \text{if}\quad \|x-\ox\|<\epsilon\; \text{ and }\; |\varphi(x)-\varphi(\ox)|<\epsilon,\\
\emp &\text{otherwise},
\end{cases}
\end{equation} 
where we skip indicating the subscript ``$\ve$" for notation simplicity. If $\varphi$ is an l.s.c.\ function, the above localization can be taken with just $\varphi(x)<\varphi(\ox)+\epsilon$ in \eqref{attentive}. Here is the aforementioned result. 

\begin{Proposition}\label{C11} Let $\varphi\colon\R^n\to\oR$ be an l.s.c.\ and prox-bounded function that is prox-regular at $\ox$ for $\ov\in\partial\varphi(\ox)$ with respect to a radius $\ve>0$. Then for all $\lambda>0$ sufficiently small, we find a convex neighborhood $U_\lambda$ of $\ox+\lambda\ov$ on which the following properties hold:

{\bf(i)} The Moreau envelope $e_\lambda\varphi$ from \eqref{Moreau} is of class ${\cal C}^{1,1}$ on the set $U_\lambda $.

{\bf(ii)} The proximal mapping $\text{\rm Prox}_{\lm\ph}$ from \eqref{ProxMapping}  is single-valued, monotone, and Lipschitz continuous on $U_\lm$ satisfying there the condition ${\rm Prox}_{\lm\ph} (x) = (I+\lambda T)^{-1}(x)$, where $T$ is the $\varphi$-attentive $\ve$-localization of $\partial \varphi$ from \eqref{attentive}. In particular, ${\rm Prox}_{\lm\ph} (\ox+\lambda \ov) = \{\ox\}$.

{\bf(iii)} The gradient of $e_\lm\ph$ is calculated by
\begin{equation}\label{GradEnvelope} 
\nabla e_\lambda\varphi(x)=\frac{1}{\lambda}\big(x-\text{\rm Prox}_{\lambda\varphi}(x)\big)=\big(\lambda I+T^{-1}\big)^{-1}(x)\;\mbox{ for all }\;x\in U_\lambda.
\end{equation}
If $\varphi$ is subdifferentially continuous at $\ox$ for $\ov$, then $T$ in \eqref{GradEnvelope} can be replaced by $\partial\varphi$. 
 \end{Proposition}
It follows from Proposition \ref{C11} that $\ox$ is a stationary point of the Moreau envelope $e_\lambda\varphi$ if and only if $\ox$ is a stationary point of $\varphi$ in the sense that $0\in\partial \varphi(\ox)$, which is known as the {\em limiting/M$($ordukhovich$)$-stationarity}. Recall that $\ox\in \dom \varphi$ is called a {\em local minimizer} to $\varphi:\R^n \rightarrow \overline{\R}$ if there is a neighborhood $U$ of $\ox$ such that 
\begin{equation*}
    \varphi (x) \ge \varphi (\ox) \ \text{ for all }\ x\in U.
\end{equation*}
The natural question arises on whether we have the {\em local minima equivalence}
\begin{equation}\label{equivalencelocalmin}
\ox \; \text{is a local minimizer of }\;\varphi \iff \; \ox \; \text{is a local minimizer of } e_\lambda\varphi.
\end{equation}
The next theorem shows that the above equivalence holds. Recall that $\ov \in \R^n$ is a {\em proximal subgradient} of $\varphi:\R^n\to\overline{\R}$ at $\ox\in\dom\ph$, denoted as $\ov\in \partial_P \varphi (\ox)$, if there are positive numbers $r$ and $\varepsilon$ for which
\begin{equation*}
\varphi(x)\ge \varphi(\ox) +\langle \ov, x-\ox\rangle  -\frac{r}{2} \|x - \ox\|^2\;\mbox{ whenever }\;x \in B_{\varepsilon}(\ox). 
\end{equation*}
It follows from Definition~\ref{prox-regfunction} that if $\varphi$ is prox-regular at $\ox$ for $\ov \in \partial\varphi(\ox)$, then $\ov$ is a proximal subgradient of $\varphi$. Note that the inverse statement is not true. Indeed, the function 
$$
\varphi(x):=\begin{cases}
|x|\left( 1+ \sin \left(\frac{1}{x} \right)\right) & \text{if }\; x \ne 0,\\
0 &\text{otherwise}
\end{cases}
$$
taken from \cite[p.~618]{Rockafellar98} is not prox-regular at $\ox$ for $\ov:=0$, but $\ov$ is a proximal subgradient of $\varphi$ at $0$. Also, if $\ox$ is a local minimizer to $\varphi$, then $0$ is a proximal subgradient of $\varphi$ at $\ox$.
  
\begin{Theorem}\label{localMoreau}  Let $\varphi: \R^n \to \overline{\R}$ be prox-bounded and l.s.c. around $\ox \in \dom \varphi$. Then for all $\lambda >0$ sufficiently small, the following are equivalent:

{\bf(i)} $\bar{x}$ is a local minimizer of $\varphi$.

{\bf(ii)} $\bar{x}$ is a local minimizer of $e_\lambda\varphi$.
\end{Theorem}
\begin{proof} 
To verify  (i)$\Longrightarrow$(ii), assume that $\ox$ is a local minimizer of $\varphi$, i.e., there exists $\varepsilon>0$ such that
\begin{equation}\label{local}
\varphi(x)\geq \varphi(\bar{x}) \quad \text{for all} \quad x \in \mathbb{B}_{\varepsilon}(\ox),
\end{equation}
which implies that $\ov=0$ is a proximal subgradient of $\varphi$ at $\ox$.  {We claim that $\text{Prox}_{\lambda \varphi}(\bar{x})=\{\bar{x}\}$ for all $\lambda>0$  sufficiently small. Indeed, the prox-boundedness of $\ph$ yields $e_{\lambda'}\varphi(\ox)>-\infty$ for some $\lambda'>0$. Then we choose $\bar{\lambda}>0$ to be sufficiently small to see that the inequalities
\begin{equation}\label{K3}
\varphi(\ox) -(2\bar{\lambda})^{-1}\|x -\ox\|^2 <  e_{\lambda'} \varphi(\ox) -(2\lambda')^{-1}\|x -\ox\|^2 \le \varphi (x)
\end{equation}
hold for all $x$ with $\|x -\ox\|>\varepsilon$. For any $0<\lambda<\bar{\lambda}$, it follows from \eqref{local} and \eqref{K3} that
\begin{equation*}
\varphi(x)+(2\lambda)^{-1}\|x-\ox\|^2 > \varphi(\ox) \;\mbox{ whenever }\;x  \neq \ox,
\end{equation*}
which is equivalent to the condition $\mathrm{Prox}_{\lambda\varphi}(\ox)=\{\ox\}$. We therefore have that $\text{Prox}_{\lambda \varphi}(\bar{x})=\{\bar{x}\}$ for all $\lambda>0$  sufficiently small.} Fixing such a small $\lambda$ and supposing that $\ox$ is not a local minimizer of $e_{\lambda}\varphi$ give us a sequence $\{x_k\}$ converging to $\ox$ for which 
\begin{equation}\label{contra-1}
e_{\lambda}\varphi (x_k) < e_{\lambda}\varphi (\ox) \text{ whenever }\;k\in \N.
\end{equation}
It follows from \cite[Theorem~1.25]{Rockafellar98} that $\text{\rm Prox}_{\lambda \varphi}(x_k)\ne\emp$, which allows us to pick $w_k\in \text{\rm Prox}_{\lambda \varphi}(x_k)$ for all $k\in \N$.  {Since $x_k\rightarrow \ox$ as $k\to\infty$ and since $\text{Prox}_{\lambda \varphi}(\bar{x})=\{\bar{x}\}$, we readily deduce from \cite[Theorem~1.25]{Rockafellar98} that $w_{k}\rightarrow \ox$ as $k\rightarrow \infty$.} 
Therefore, $w_{k}\in \mathbb{B}_{\varepsilon}(\ox)$ for sufficiently large $k\in \N$, which leads us by \eqref{local} to the relationships
\begin{eqnarray*}
e_\lambda\varphi(\bar{x})\leq\varphi(\bar{x})\le\varphi(w_{k})
=e_\lambda\varphi(x_{k})-\frac{1}{2\lambda}\|x_{k}-w_{k}\|^2,
\end{eqnarray*}
which clearly contradict \eqref{contra-1}. Thus $\bar{x}$ is a local minimizer of $e_\lambda\varphi$ for all $\lambda>0$ sufficiently small.\vspace*{0.03in}

To verify the opposite implication, assume that (ii) holds and get $\text{\rm Prox}_{\lambda\varphi}(\ox)\ne\emp$ by \cite[Theorem~1.25]{Rockafellar98}. Moreover, the l.s.c.\ and prox-boundedness of $\varphi$ imply due to \cite[Example~10.32]{Rockafellar98} that
\begin{equation}\label{sub-e}
\partial e_{\lambda}\varphi (\ox) \subset \dfrac{1}{\lambda}\big[\ox-\text{\rm Prox}_{\lambda\varphi}(\ox)\big].
\end{equation}
By the generalized Fermat rule for limiting subgradients in \cite[Proposition~1.114]{Mordukhovich06}, we have $0\in \partial e_{\lambda}\varphi (\ox)$. Combining the latter with \eqref{sub-e} ensures that 
\begin{equation*}\label{eq:Fermat-e}
0\in \partial e_{\lambda}\varphi (\ox)\subset \dfrac{1}{\lambda}\big[\ox-\text{\rm Prox}_{\lambda\varphi}(\ox)\big],
\end{equation*}
and so $\ox \in \text{\rm Prox}_{\lambda\varphi}(\ox)$. The minimization in (ii) yields the existence of a neighborhood $U$ of $\ox$ with
$$
\varphi(x) \geq e_\lambda\varphi(x) \geq e_\lambda\varphi(\ox)= \varphi(\ox) + \frac{1}{2\lambda}\|\ox-\ox\|^2 = \varphi(\ox) \quad \text{whenever} \; x \in U,
$$
which therefore completes the proof of the theorem.
\end{proof}\vspace*{-0.25in}

\section{Second-Order Necessary Conditions under Prox-Regularity}\label{sec:optiforprox}\vspace*{-0.05in}

This section is devoted to deriving second-order necessary optimality conditions in the unconstrained extended-real-valued format of minimization with  general prox-regular objectives. The tools of variational analysis used for these purposes are {\em second-order subdifferentials}, which are defined below in the scheme of \cite{Mor92} as coderivatives of subgradient mappings. The following second-order constructions of this type are described via combinations of the subgradients and coderivatives discussed in Section~\ref{sec:prelim}.

\begin{Definition}\label{2ndsub}\rm 
Let $\varphi: \R ^n \to \overline{\R }$  be a function finite at some point $\bar{x}$. 

{\bf(i)} For any $\bar{v} \in \partial \varphi(\bar{x})$, the multifunction $\partial^2 \varphi(\bar{x},\bar{v}): \R ^n \rightrightarrows \R ^n$ with the values
$$
\partial^2 \varphi(\bar{x},\bar{v})(u):= (D^*\partial\varphi)(\bar{x},\bar{v})(u), \quad u \in \R ^n,
$$
is the {\em limiting/basic second-order subdifferential}  of $\varphi$ at $\bar{x}$ relative to $\bar{v}$. 

{\bf(ii)} For any $\bar{v} \in {\partial} \varphi(\bar{x})$, the multifunction $\breve{\partial}^2 \varphi(\bar{x},\bar{v}): \R ^n \rightrightarrows \R ^n$ with the values
$$
\breve{\partial}^2 \varphi(\bar{x},\bar{v})(u):= (\widehat{D}^* {\partial}\varphi)(\bar{x},\bar{v})(u), \quad u \in \R ^n,
$$
is the {\em combined second-order subdifferential} of $\varphi$ at $\bar{x}$ relative to $\bar{v}$. 

{\bf(iii)} For any $\bar{v} \in\Hat\partial\varphi(\bar{x})$, the multifunction $\widehat{\partial}^2 \varphi(\bar{x},\bar{v}): \R ^n \rightrightarrows \R ^n$ with the values
$$
\widehat{\partial}^2 \varphi(\bar{x},\bar{v})(u):= (\widehat{D}^* {\widehat{\partial}}\varphi)(\bar{x},\bar{v})(u),\quad u \in \R^n,
$$
is the {\em regular second-order subdifferential}  of $\varphi$ at $\bar{x}$ relative to $\bar{v}$. 
\end{Definition}
We skip $\bar{v} = \nabla \varphi(\bar{x})$ in the above second-order subdifferentials if $\varphi$ is $\mathcal{C}^1$-smooth around $\bar{x}$. The inclusions given below are obvious:
\begin{equation*}
\breve{\partial}^2\varphi(x,v)(u)\subset\partial^2\varphi(x,v)(u)\;\text{ for all }\;(x,v)\in\text{\rm gph}\,\partial\varphi,\;u\in\R^n.
\end{equation*}
\begin{equation*}
\breve{\partial}^2\varphi(x,v)(u)\subset\widehat{\partial}^2\varphi(x,v)(u)\;\text{ for all }\; (x,v)\in\text{\rm gph}\,\partial\varphi,\;u\in\R^n.
\end{equation*}
In general, the regular second-order subdifferential and the limiting one are incomparable. However, if $\varphi\in\mathcal{C}^1$ around $\bar{x}$, then we have the relationships
\begin{equation}\label{Frechet_Mordukhovich}
\breve{\partial}^2 \varphi(\bar{x})(u)= \widehat{\partial}^2 \varphi(\bar{x})(u)\subset\partial^2 \varphi(\bar{x})(u)\;\mbox{for all }\;u \in \R^n.
\end{equation}
If $\varphi\in\mathcal{C}^{1,1}$ around $\bar{x}$, then the second-order subdifferential calculation can be simplified by the scalarization formulas below taken from \cite[Proposition~3.5]{bny07} and \cite[Proposition~1.120]{Mordukhovich06}, respectively:
\begin{equation*}
\widehat{\partial}^2 \varphi(\bar{x})(u)=\breve{\partial}^2 \varphi(\bar{x})(u)=\widehat\partial\langle u,\nabla\varphi\rangle(\bar{x}), \quad
{\partial}^2 \varphi(\bar{x})(u)=\partial\langle u,\nabla\varphi\rangle(\bar{x}),\quad u\in\R^n.
\end{equation*}
If $\varphi$ is $\mathcal{C}^2$-smooth around $\bar{x}$, then the above second-order subdifferentials reduce to the classical Hessian
\begin{equation*}
\partial^2 \varphi(\bar{x})(u)=\breve{\partial}^2 \varphi(\bar{x})(u)=\widehat{\partial}^2 \varphi(\bar{x})(u)= \{\nabla^2 \varphi(\bar{x})u \} \quad\mbox{whenever }\;u \in \R ^n.
\end{equation*} 

To derive second-order necessary optimality conditions for local minimizers of prox-regular functions, we begin with the following useful remark.

\begin{Remark}\label{prox-bounded} Let $\varphi:\R^n\to\overline{\R}$ be prox-regular at $\bar{x}\in \dom \varphi$ for $\ov \in\partial\varphi(\bar{x})$ with respect to $\ve>0$ and $r\ge 0$. Then the function $\psi:\R^n\to \overline{\R}$ given by
\begin{equation}\label{proxboundedness}
\psi(x):= \varphi (x)+\delta_{B_\varepsilon(\ox)}(x)\;\text{ whenever }\;x\in\R^n
\end{equation} 
is prox-bounded and prox-regular at $\ox$ for $\ov \in \partial\psi(\ox)$ with respect to the same $\ve>0$ and $r\ge 0$. Furthermore, we have 
\begin{equation}\label{chain}
\breve{\partial}^2 \psi(x,y) (u) = \breve{\partial}^2\varphi(x,y)(u)\;\mbox{ for all }\;(x,y) \in\gph\partial \psi \cap (B_\varepsilon(\ox)\times \R^n)\;\mbox{ and }\;u \in \R^n. 
\end{equation}
\end{Remark}

The next theorem provides the desired {\em second-order subdifferential necessary optimality condition} for local minimizers of the general class of prox-regular functions.
 
\begin{Theorem}\label{NecessI} If $\ox\in\dom\ph$ is a local minimizer of an extended-real-valued function $\ph\colon\R^n\to\oR$, then $0\in\partial\ph(\ox)$. Assuming in addition that $\ph$ is prox-regular at $\ox$ for $0$, we have the pointbased second-order necessary optimality condition for local minimizers
\begin{equation}\label{PSDnecessary}
\langle z,w\rangle \geq 0\;\text{ whenever }\;z \in \breve{\partial}^2\varphi(\bar{x},0)(w)\;\mbox{ and }\;w\in \R^n
\end{equation} 	
expressed via the combined second-order subdifferential of $\ph$ at $\ox$ relative to $0$ from Definition~{\rm\ref{2ndsub}(ii)}.
\end{Theorem}
\begin{proof} The stationary condition $0 \in \partial\varphi(\ox)$ for a local minimizer $\ox$ follows immediately from \cite[Proposition~1.30]{Mor18}. By Remark~\ref{prox-bounded}, one finds an $\varepsilon>0$ such that $\psi:=\varphi+\delta_{B_{\ve}(\ox)}$ from \eqref{proxboundedness} is prox-bounded and prox-regular at $\ox$ for $0\in \partial \psi (\ox)$ with respect to such radius $\ve$. Obviously, $\psi$ is l.s.c. around $\ox$ as well. It follows from Proposition~\ref{C11} that there exists a neighborhood $U_\lambda$ of $\bar{x}$ for all $\lm>0$ sufficiently small such that the Moreau envelope $e_\lambda \psi$ is $\mathcal{C}^{1,1}$ on $U_\lambda$ and that representation \eqref{GradEnvelope} is satisfied in terms of $\psi$. Theorem~\ref{localMoreau} tells us that $\bar{x}$ is also a local minimizer of $e_\lambda\psi$ for all small $\lm>0$. Then it follows from \cite[Theorem~3.3]{ChieuLeeYen17} that the multifunction $\breve{\partial}^2e_\lambda\psi(\bar{x})$ is positive-semidefinite, i.e., 
\begin{equation}\label{PSDMoreaunecess}
\langle z,u\rangle \geq 0\;\text{ whenever }\;z \in\breve{\partial}^2e_\lambda\psi(\bar{x})(u),\quad u\in \R^n.
\end{equation}
Pick any $w\in \R^n$ and $z \in\breve{\partial}^2\varphi(\bar{x},0)(w)$. It follows from \eqref{chain} that $z\in \breve{\partial}^2 \psi (\ox,0)(w)$. Since $\gph T \subset \gph \partial \psi$, we have the inclusion
$$
\widehat{N}_{{\rm\small gph}\,\partial\psi} (\ox,0) \subset \widehat{N}_{{\rm\small gph}\,T}(\ox,0),
$$
which tells us by the definition of the regular coderivative that
\begin{equation}\label{subsetbreve}
\breve{\partial}^2\psi(\bar{x},0)(w)=(\widehat{D}^*\partial \psi)(\bar{x},0)(w) \subset \widehat{D}^*T(\bar{x},0)(w).
\end{equation}
It follows from \eqref{subsetbreve} and $z\in \breve{\partial}^2\psi(\bar{x},0)(w)$ that $-w\in\widehat{D}^*(T^{-1})(0,\bar{x})(-z)$. Using further \eqref{GradEnvelope} and then applying the coderivative sum rule from \cite[Theorem 1.62]{Mordukhovich06} lead us to the equalities
$$
\widehat{D}^*\big(\nabla e_\lambda \psi\big)^{-1}(0,\bar{x})(-z) =\widehat{D}^*(\lambda I + T^{-1})(0,\bar{x}\big)(-z) = -\lambda z + \widehat{D}^*(T^{-1})(0,\bar{x})(-z),$$
which imply in turn the inclusion
$$
-\lambda z  - w \in\widehat{D}^*(\nabla e_\lambda \psi)^{-1}(0,\bar{x})(-z).
$$
This tells us that $z\in\widehat{D}^*(\nabla e_\lambda \psi)(\bar{x})(\lambda z + w)$ meaning that $z \in \breve{\partial}^2e_\lambda\psi(\bar{x})(\lambda z + w)$. By \eqref{PSDMoreaunecess}, we get therefore that $\langle z, \lambda z + w\rangle \geq 0$. Letting $\lambda \dn 0$ in the latter inequality, we arrive at $\langle z, w\rangle \geq 0$ and thus justify the claimed condition \eqref{PSDnecessary}, which completes the proof of the theorem.
\end{proof}

\begin{Remark} \rm 
Note that we {\em cannot replace} the combined second-order subdifferential in \eqref{PSDnecessary} by the limiting one, i.e., the pointbased  relationship
\begin{equation}\label{PSDof2ndlimiting}
\langle z, w\rangle \geq 0\;\text{ for all }\;z \in {\partial}^2\varphi(\bar{x},0)(w)\;\mbox{ and }\;w\in \R^n
\end{equation}
fails to be a necessary condition for the local optimality of $\ox$ in minimizing a prox-regular function $\ph$. Indeed, it is discussed in \cite[Example~4.14]{dima} that there is an example in \cite{huytuyen16} showing that \eqref{PSDof2ndlimiting} fails for a local minimizer $\ox$ of a function $\ph\colon\R^2\to\oR$ that is {\em fully amenable} at $\ox$ being surely continuously prox-regular at $\ox$ for $0\in\partial\ph(\ox)$; see \cite{Rockafellar98}. Furthermore, Example~3.2 in \cite{ChieuLeeYen17} illustrates the violation of \eqref{PSDof2ndlimiting} for a local minimizer $\ox$ of a function $\ph\colon\R\to\R$ from class ${\cal C}^{1,1}$ around 
$\ox$, which is also continuously prox-regular at $\ox$ for $0=\nabla\ph(\ox)$.
\end{Remark}\vspace*{-0.2in}

\section{Variational Convexity and Second-Order Sufficient Conditions}\label{sec:vc}\vspace*{-0.05in}

Our goal here is to derive a {\em neighborhood} counterpart of the second-order subdifferential necessary condition \eqref{PSDnecessary} that are {\em sufficient} for local optimality in the extended-real-valued framework of prox-regular functions. In fact, we will do more. Namely, the main result of this section provides a new {\em coderivative characterization} of the {\em variational convexity} property for general l.s.c.\ functions $\ph\colon\R^n\to\oR$, which yields the sufficiency of local minimizers. 

We start with the definition of variational convexity (and its strong version used in the next section). These properties have been recently introduced by Rockafellar \cite{r19} and then have been further investigated and applied in \cite{KKMP23-2,kmp22convex,roc,ding}.

\begin{Definition}\label{vr} \rm An l.s.c.\  function $\varphi:\R^n\to\overline{\R}$ is {\em variationally convex} at $\ox$ for $\ov\in\partial\varphi(\ox)$ if there exist a convex neighborhood $U\times V$ of $(\ox,\ov)$ and an l.s.c.\ convex function $\psi\le\varphi$ on $U$ such that 
\begin{equation}\label{varconvex}
(U_\epsilon\times V)\cap\gph\partial\varphi=(U\times V)\cap\gph\partial\psi\;\text{and}\;\varphi(x)=\psi(x)\;\text{at the common elements}\;(x,v)
\end{equation}
for some $\ve>0$, where $U_\epsilon:=\{x\in U\;|\;\varphi(x)<\varphi(\ox)+\epsilon\}$. We say that $\ph$ is {\em strongly variationally convex} at $\ox$ for $\ov$ with modulus $\sigma >0$ if \eqref{varconvex} holds with $\psi$ being strongly convex on $U$ with this modulus.
\end{Definition}

First we provide a coderivative characterization of the {\em cocoercivity property} of single-valued continuous mappings. This result is of its own importance while being useful to derive the aforementioned characterizations of variational convexity.

\begin{Proposition}\label{prop:coco}
Let $f: U \rightarrow \R^n$ be a continuous mapping defined on an open set $U\subset \R^n$, and let $\sigma>0$. Then the following assertions are equivalent:

{\bf(i)} $f$ is cocoercive on $U$ with modulus $\sigma$, i.e., for arbitrary $x_1,x_2\in U$ it holds that
\begin{equation}\label{cocoercive}
\la f (x_1)-f (x_2),x_1-x_2\ra \ge \sigma \|f (x_1)-f (x_2)\|^2.
\end{equation}

{\bf(ii)} We have the coderivative  estimate
\begin{equation*}
\la w,z\ra \ge \sigma \|w\|^2 \ \text{ whenever }\ w\in D^* f (x)(z),\ z\in \R^n,\; \text{ and }\ x\in U.
\end{equation*}

{\bf(iii)} We have the coderivative estimate
\begin{equation}\label{2nd-order-posi}
\la w,z\ra \ge \sigma \|w\|^2 \ \text{ whenever }\ w\in \widehat{D}^* f (x)(z),\ z\in \R^n,\;\text{ and }\ x\in U.
\end{equation}
\end{Proposition}
\begin{proof}
To verify implication (i)$\Longrightarrow$(ii), pick any $x\in U$, $z\in \R^n$, and $w\in D^*f(x)(z)$. By applying the Cauchy-Schwarz inequality to (i), we clearly get the Lipschitz continuity of $f$ on $U$ with modulus $\frac{1}{\sigma}$. Define the mapping $g_z\colon U\to\R$ by 
\begin{equation*}
g_z(y):=\big\la z,2\sigma f(y)-y\big\ra\;\text{ for all }\; y\in U.
\end{equation*}
We aim to show that $g_z$ is Lipschitz continuous on $U$ with modulus $\|z\|$ by taking arbitrary $x_1,x_2\in U$ and employing
the cocoercive condition \eqref{cocoercive} for them. This yields
\begin{equation*}
\left\|2 \sigma\left[f\left(x_1\right)-f\left(x_2\right)\right]-(x_1-x_2)\right\| \leq\left\|x_1-x_2\right\|,
\end{equation*}
and thus $\big|\left\langle z, 2 \sigma\left[f\left(x_1\right)-f\left(x_2\right)\right]-(x_1-x_2)\right\rangle\big| \leq \|z\|\cdot \left\|x_1-x_2\right\|$. By the definition of $g_z$, we get
\begin{equation*}
\left|g_z (x_1)-g_z (x_2)\right|\le \|z\|\cdot \|x_1-x_2\|,
\end{equation*}
which verifies the Lipschitz continuity of $g_z$ on $U$ with modulus $\|z\|$. Then the subdifferential estimate for Lipschitz continuous functions taken from \cite[Corollary~1.81]{Mordukhovich06} ensures that
\begin{equation}\label{Lip-sub}
\|x^*\|\le \|z\| \ \text{ for all }\ x^* \in \partial g_z (x).  
\end{equation}
Taking into account the structure of the function $g_z$ and the Lipschitz continuity of $f$ allows us to employ the limiting subdifferential sum rule from \cite[Theorem~1.107(ii)]{Mordukhovich06} and the scalarization formula for the limiting coderivative from \cite[Theorem~1.90]{Mordukhovich06} to obtain the relationships
\begin{eqnarray*}
2\sigma w-z &\in&  2\sigma D^* f(x)(z) +\nabla \la -z,\cdot\ra (x) = D^* (2\sigma f)(x)(z)+\nabla\la -z,\cdot\ra (x) = \partial\la z,2\sigma f\ra (x)+\nabla\la -z,\cdot\ra (x) \\
&=& \partial\la z,2\sigma f-\cdot\ra (x) = \partial g_z (x),
\end{eqnarray*}
bringing us to $2\sigma w-z \in \partial g_z (x)$. It follows from the basic subgradient estimate \eqref{Lip-sub} that $\|2\sigma w-z\|\le \|z\|$, which amounts to $4\sigma^2 \|w\|^2 -4\sigma \la w,z\ra \le 0$, i.e., $\la w,z\ra \ge \sigma \|w\|^2$, and thus (ii) holds. 

Implication (ii)$\Longrightarrow$(iii) is trivial since $\widehat{D}^*f(x)(z)\subset D^* f(x)(z)$ for all $x\in U$ and $z\in \R^n$.

Let us proceed next with verifying (iii)$\Longrightarrow$(i). Assuming the coderivative estimate \eqref{2nd-order-posi}, we intend to show that $f-\frac{1}{2\sigma}I$ is Lipschitz continuous on $U$ with modulus $\frac{1}{2\sigma}$, where $I$ is the identity operator on $\R^n$. To this end, pick any $x\in U$, $z\in \R^n$, and $w\in \widehat{D}^* \left(f-\frac{1}{2\sigma}I\right)(x)(z)$. Employing the aforementioned sum rule for the regular coderivative tells us that
\begin{equation*}
w\in \widehat{D}^* \left(f-\frac{1}{2\sigma}I\right)(x)(z) = \widehat{D}^* f(x)(z)-\frac{1}{2\sigma}z, 
\end{equation*}
which yields $w+\frac{1}{2\sigma}z\in \widehat{D}^* f(x)(z)$. Then we deduce from \eqref{2nd-order-posi} that
\begin{equation*}
\Big\la w+\frac{1}{2\sigma}z,z\Big\ra \ge \sigma \left\|w+\dfrac{1}{2\sigma}z\right\|^2,\text{ i.e., } \dfrac{\|z\|^2}{4\sigma}\ge \sigma \|w\|^2,
\end{equation*}
which is equivalent to $\|w\|\le \dfrac{1}{2\sigma}\|z\|$. It  follows from the regular coderivative characterization of the Lipschitz continuity in \cite[Theorem~4.7]{Mordukhovich06} that the mapping $f-\frac{1}{2\sigma}I$ is Lipschitz continuous on $U$ with modulus $\frac{1}{2\sigma}$. Using this and applying  \cite[Proposition~2.3]{Zhu96} verify the cocoercivity property  of $f$ in \eqref{cocoercive}, which gives us (i) and thus completes the proof of the proposition.
\end{proof}

\begin{Remark}\rm 
 {It is pointed out to us by one of the referees that Proposition \ref{cocoercive} can be deduced from \cite[Lemma~3.3~and~Theorem~3.4]{Nghia16} via the fact that $f$ is cocoercive if and only if $f^{-1}$ is strongly monotone. Still, we present a direct proof of this proposition for the sake of self-containment.}
\end{Remark}

Next we present a sufficient condition for the cocoercivity of the proximal mappings.

\begin{Proposition}\label{coro:cocoprox}
Let $\varphi:\R^n \rightarrow \overline{\R}$ be prox-bounded and prox-regular at $\ox\in \dom \varphi$ for $\ov\in \partial \varphi (\ox)$ with respect to a radius $\ve>0$. Suppose that there exists a neighborhood $W$ of $(\ox,\ov)$ for which  
\begin{eqnarray}\label{eq:2nd-D^}
\la z,w\ra \ge 0\ \text{ whenever }\ z\in \Hat{D}^* T(u,v)(w),\ w\in \R^n\ \text{ and } \ (u,v)\in \gph T \cap W,
\end{eqnarray}
where $T$ is the $\varphi$-attentive $\ve$-localization of $\partial\varphi$ around $(\ox,\ov)$ taken from Proposition \ref{C11}. Then for any sufficiently small $\lambda>0$ the proximal mapping $\text{\rm Prox}_{\lambda\varphi}$ is cocoercive with modulus $1$ on a neighborhood of $\ox+\lambda \ov$.
\end{Proposition}
\begin{proof}
By Proposition~\ref{C11}, for any $\lambda>0$ sufficiently small, there exists a neighborhood $U_{\lambda}$ of $\ox+\lambda \ov$ on which assertions (i)--(iii) of Proposition~~\ref{C11} hold in terms of $T$. Fixing such a small number $\lambda$, we claim that $\text{\rm Prox}_{\lambda\varphi}$ is cocoercive with modulus $\sigma=1$ on $U_{\lambda}$ by employing Proposition \ref{prop:coco}. Note first the continuity of ${\rm Prox}_{\lm\ph}$ on $U_{\lambda}$ yields this property for the mapping $\Phi: U_{\lambda}\rightarrow\R^n\times\R^n$ defined by
\begin{eqnarray*}
y\mapsto\left( {\rm Prox}_{\lm\ph}(y),\dfrac{y-{\rm Prox}_{\lm\ph}(y)}{\lambda}\right).
\end{eqnarray*}
Observe that $\Phi (\ox+\lambda \ov) = (\ox,\ov)\in W$ and then shrink if necessary the neighborhood $U_{\lambda}$ of $\ox +\lambda \ov$ so that $\Phi (U_{\lambda})\subset W$. Picking any $y\in U_{\lambda}$, let $x:= {\rm Prox}_{\lm\ph}(y)$ and $z\in \widehat{D}^*{\rm Prox}_{\lm\ph}(y)(w)$ for $w\in \R^n$. This implies by Proposition~\ref{C11}(ii) and the coderivative sum rule from \cite[Theorem~1.62(i)]{Mordukhovich06} that
\begin{eqnarray*}
-w \in \widehat{D}^* ({\rm Prox}_{\lm\ph})^{-1}(x,y)(-z) &=& \widehat{D}^* (I+\lambda T)(x,y)(-z) = -z + \widehat{D}^*(\lambda T)(x,y-x)(-z) \\
&=& -z + \lambda \widehat{D}^* T\left(x,\dfrac{y-x}{\lambda}\right)(-z). 
\end{eqnarray*}
Therefore, we arrive at the coderivative inclusion
\begin{equation}\label{use-2nd}
\dfrac{z-w}{\lambda}\in \widehat{D}^* T\left(x,\dfrac{y-x}{\lambda}\right)(-z). 
\end{equation}
Since $\left(x,\dfrac{y-x}{\lambda}\right) = \Phi (y) \in W$, it follows from \eqref{2nd-D^} and \eqref{use-2nd} that $\big\la \frac{z-w}{\lambda},-z\big\ra \ge 0$, i.e., 
\begin{equation}\label{eq:2nd-growth}
\la w,z\ra \ge \|z\|^2\;\text{ whenever }\;z\in \widehat{D}^* \text{\rm Prox}_{\lambda\varphi}(y)(w).
\end{equation}
Applying now Proposition~\ref{prop:coco} to the continuous mapping $\text{\rm Prox}_{\lambda\varphi}$ on $U_{\lambda}\ni \ox + \lambda \ov$ satisfying the coderivative estimate \eqref{eq:2nd-growth} tells us that $\text{\rm Prox}_{\lambda\varphi}$ is cocoercive with modulus $\sigma=1$ on $U_{\lambda}$. 
\end{proof}

Now we are ready to establish the new coderivative-based  {\em characterizations of variational convexity}. Note that characterizations of this type first appeared in \cite{kmp22convex} while validating subdifferential continuity.

\begin{Theorem}\label{theo:codforvc}
Let $\varphi:\R^n \rightarrow \overline{\R}$ be an l.s.c. function with $\ox\in \dom \varphi$ and $\ov\in \partial \varphi (\ox)$. Then the following assertions are equivalent:

{\bf(i)} $\varphi$ is variationally convex at $\ox$ for $\ov$.

{\bf(ii)} $\varphi$ is prox-regular at $\ox$ for $\ov$ with respect to a radius $\ve>0$, and there is a neighborhood $W$ of $(\ox,\ov)$ such that 
\begin{eqnarray}\label{2nd-D*}
\la z,w\ra \ge 0\ \text{ whenever }\ z\in D^*T^{\varphi}(u,v)(w)\ \text{ and } \ (u,v)\in \gph T^{\varphi} \cap W.
\end{eqnarray}

{\bf(iii)} $\varphi$ is prox-regular at $\ox$ for $\ov$ with respect to a radius $\ve>0$, and there is a neighborhood $W$ of $(\ox,\ov)$ such that 
\begin{eqnarray}\label{2nd-D^}
\la z,w\ra \ge 0\ \text{ whenever }\ z\in \Hat{D}^* T^{\varphi}(u,v)(w)\ \text{ and } \ (u,v)\in \gph T^{\varphi} \cap W,
\end{eqnarray}
where $T^{\varphi}$ is the $\varphi$-attentive $\varepsilon$-localization of $\partial\varphi$ around $(\ox,\ov)$. If in addition we suppose that $\varphi$ is subdifferentially continuous at $\ox$ for $\ov$, then the coderivatives $D^* T^{\varphi}(u,v)$ and $\Hat{D}^* T^{\varphi}(u,v)$ in \eqref{2nd-D*} and \eqref{2nd-D^} can be replaced by the second-order subdifferentials $\partial^2 \varphi(u,v)$ and $\breve{\partial}^2\varphi(u,v)$, respectively. 
\end{Theorem}
\begin{proof}
To verify implication (i)$\Longrightarrow$(ii), we get by Definition~\ref{vr} of variational convexity a number $\varepsilon>0$, a neighborhood $U\times V$ of $(\ox,\ov)$, and a convex l.s.c.\ function $\psi:\R^n \rightarrow \overline{\R}$ such that
\begin{equation}\label{gph-vc}
(U_{\varepsilon}\times V)\cap \gph \partial \varphi = (U\times V)\cap \gph \partial \psi
\end{equation}
with $U_{\varepsilon}$ defined therein. We shrink $U,V$, if necessary, so that $U\subset B_\ve(\ox)$ and $V\subset B_\ve(\ov)$. Let $T^{\varphi}$ be the $\varphi$-attentive $\varepsilon$-localization of $\partial \varphi$ around $(\ox,\ov)$ taken from \eqref{attentive}. Then we have the inclusion $(U_{\varepsilon}\times V) \cap \gph \partial \varphi \subset \gph T^{\varphi}$. Pick further $(u,v)\in (U\times V)\cap \gph T^{\varphi}$, $w\in \R^n$ and $z\in D^*T^{\varphi}(u,v)(w)$. Then it follows from definition \eqref{lcod} of the limiting coderivative that there are sequences $z_k\rightarrow z$, $w_k\rightarrow w$, $(u_k,v_k)\xrightarrow{\mathrm{gph} T^{\varphi}\cap (U\times V)}(u,v)$ for which $z_k \in \widehat{D}^*T^{\varphi} (u_k,v_k)(w_k)$ for all $k\in \N$ large enough. Using the definitions of $T^{\varphi}$, $U$, and $V$ implies that $(u_k,v_k)\in (U_{\varepsilon}\times V)\cap \gph \partial \varphi = (U\times V)\cap \gph \partial \psi$. Consequently, the choice of $z_k\in \widehat{D}^*T^{\varphi}(u_k,v_k)(w_k)$ ensures that 
\begin{equation*}
(z_k,-w_k)\in \widehat{N}_{\mathrm{gph}\,T^{\varphi}}(u_k,v_k) \subset \widehat{N}_{(U_{\varepsilon}\times V)\cap \mathrm{gph}\, \partial \varphi}(u_k,v_k) = \widehat{N}_{(U\times V)\cap \mathrm{gph}\,\partial \psi}(u_k,v_k) = \widehat{N}_{\mathrm{gph}\, \partial \psi}(u_k,v_k),
\end{equation*}
and hence $z_k\in \widehat{D}^* (\partial\psi)(u_k,v_k)(w_k) = \breve{\partial}^2 \psi (u_k,v_k)(w_k)$. Then we deduce from \cite[Theorem~3.2]{CCYY} that $\la z_k,w_k\ra \ge 0$ for all $k\in\N$ sufficiently large, which thus completes the proof of implication (i)$\Longrightarrow$(ii) by passing to the limit as $k\rightarrow \infty$. The next implication (ii)$\Longrightarrow$(iii) is trivial since the regular coderivative is always contained in the corresponding limiting one.

Finally, we check implication (iii)$\Longrightarrow$(i). Considering the localization function $\psi:=\varphi+\delta_{B_{\ve}(\ox)}$, we get by Remark \ref{prox-bounded} that $\psi$ is prox-bounded and prox-regular at $\ox$ for $\ov\in \partial\psi (\ox)$ with respect to the radius $\ve>0$. Further, one may easily check that $\gph T^{\varphi} = \gph T^{\psi}$, where 
\begin{equation*}
    \gph T^{\psi}:= \{(x,v)\in \gph \partial \psi \mid \|x-\ox\|<\ve,\|v-\ov\|<\ve \text{ and } |\varphi (x) - \varphi (\ox)|<\ve\}.
\end{equation*}
This implies via \eqref{2nd-D^} that 
\begin{equation}\label{2nd-D}
    \la z,w\ra \ge 0\ \text{ whenever }\ z\in \Hat{D}^* T^{\psi}(u,v)(w)\ \text{ and } \ (u,v)\in \gph T^{\psi} \cap W.
\end{equation}
Then Proposition \ref{C11} and Proposition \ref{coro:cocoprox} together ensure that for any small $\lambda>0$, there is a neighborhood $U_{\lambda}$ of $\ox+\lambda \ov$ on which assertions (i)--(iii) of Proposition~\ref{C11} hold in terms of $\psi$, and moreover, the proximal mapping $\text{\rm Prox}_{\lambda\psi}$ is cocoercive with modulus $1$ on $U_{\lambda}$.

Next we claim that the subgradient mapping $\partial\varphi$ is locally $\varphi$-monotone around $(\ox,\ov)$ in the sense that there exists a neighborhood $U\times V$ of $(\ox,\ov)$ such that $\partial\ph$ is monotone on $U_\ve\times V$, where $U_\ve$ is 
taken from Definition~\ref{vr}. To proceed, define the homeomorphism $\Theta_{\lambda}: \R^n \times \R^n \rightarrow \R^n \times \R^n$ by $(x,v)\mapsto (x+\lambda v,x)$ and observe that
\begin{equation*}
\Theta_{\lambda}(\gph T^{\psi}) =\gph (I+\lambda T^{\psi})^{-1}\;\mbox{ and }\;\Theta_{\lambda}^{-1}(\ox+\lambda \ov,\ox) = (\ox,\ov)
\end{equation*} 
with $\Theta_{\lambda}^{-1}(U_{\lambda}\times \R^n)\cap \big(B_{\varepsilon}(\ox)\times B_{\varepsilon}(\ov)\big)$ being
a neighborhood of $(\ox,\ov)$. Select two pairs $(u_i,v_i)\in \gph \partial \varphi\cap \Theta_{\lambda}^{-1}(U_{\lambda}\times \R^n)\cap \big(B_{\varepsilon}(\ox)\times B_{\varepsilon}(\ov)\big)$ such that $\varphi(u_i)<\varphi (\ox) + \varepsilon$, which ensures that $(u_i,v_i)\in \gph T^{\psi}$ for $i=1,2$. By $\text{\rm Prox}_{\lambda\psi} = (I+\lambda T^{\psi})^{-1}$ on $U_{\lambda}$, we have the equalities
\begin{eqnarray*}
(u_i + \lambda v_i,u_i) &=& \Theta_{\lambda}(u_i,v_i)\in \Theta_{\lambda} \big( \gph T^{\psi} \cap \Theta_{\lambda}^{-1}(U_{\lambda}\times \R^n)\big) =\Theta_{\lambda} (\gph T^{\psi}) \cap \Theta_{\lambda}\big(\Theta_{\lambda}^{-1}(U_{\lambda}\times \R^n)\big) \\
&=& \gph (I+\lambda T^{\psi})^{-1} \cap \big(U_{\lambda}\times \R^n\big)  =\gph \text{\rm Prox}_{\lambda\psi} \cap (U_{\lambda}\times \R^n),\quad i=1,2.
\end{eqnarray*}
Since the mapping $\text{\rm Prox}_{\lambda\psi}$ is cocoercive with modulus $1$ on $U_{\lambda}$, it follows that
\begin{equation*}
\la u_1-u_2,u_1 - u_2 +\lambda (v_1-v_2)\ra \ge \|u_1-u_2\|^2.
\end{equation*}
This yields $\la u_1-u_2,v_1-v_2\ra \ge 0$ and thus justifies the local $\varphi$-monotonicity of $\partial \varphi$ around $(\ox,\ov)$. Applying now \cite[Theorem~6.6]{KKMP23-2} tells us that $\ph$ is variationally convex at $\ox$ for $\ov$.

It remains to show that $D^*T(u,v)$ and ${\Hat D}^*T(u,v)$ can be replaced by $\partial^2 \varphi(u,v)$ and $\breve\partial^2 \varphi(u,v)$, respectively, provided that $\ph$ is subdifferentially continuous at $\ox$ for $\ov$. Indeed, the latter property allows us to shrink $W$ if necessary so that $\gph T = \gph \partial\varphi\cap W$. This tells us that 
$$\partial^2 \varphi (u,v) = D^* T(u,v) \ \text{ and }\ 
\breve{\partial}^2\varphi(u,v) = \widehat{D}^*T(u,v)\;\text{for all }\; (u,v) \in W
$$
and therefore completes the proof of the theorem. 
\end{proof}

\begin{Remark}\rm
As mentioned prior to Theorem \ref{theo:codforvc}, second-order subdifferential characterizations of variational convexity are obtained in \cite[Theorem~5.2]{kmp22convex} under the subdifferential continuity assumption. Observe that the approach employed in deriving the result of \cite[Theorem~5.2]{kmp22convex} can be adapted to provide an alternative proof of Theorem~\ref{theo:codforvc}. The key tools in both approaches leverage the prox-regularity of the given function revolve around establishing well-posed properties of the corresponding Moreau envelope and proximal mapping. In both proofs, the ultimate goal is to verify the local convexity of the Moreau envelope, which yields the variational convexity of the function in question via \cite[Theorem~3.2]{kmp22convex}. Nevertheless, the local convexity of $e_{\lambda}\varphi$ is achieved in these devices differently. Indeed, the proof above employs Proposition~\ref{prop:coco} to justify the cocoercivity of $\text{\rm Prox}_{\lambda}\varphi$, which leads us to the $\varphi$-local monotonicity of $\partial\varphi$ and then to the local convexity of $e_{\lambda}\varphi$ by \cite[Theorem~6.6]{KKMP23-2} under imposing prox-regularity. On the other hand, the proof in \cite[Theorem~5.2]{kmp22convex} verifies the local convexity of $e_{\lambda}\varphi$ via the second-order regular subdifferential condition from \cite[Lemma~5.1]{kmp22convex}.
\end{Remark}\vspace*{-0.05in}

As a consequence of Theorem~\ref{theo:codforvc}, we arrive at the following new {\em second-order subdifferential sufficient conditions} for local minimizers of extended-real-valued prox-regular functions.

\begin{Corollary}\label{cor:2nd-suff} Let $\varphi:\R^n\to\overline{\R}$ be prox-regular at $\bar{x}\in\dom\varphi$ for $0\in\partial\varphi(\bar{x})$ with respect to a radius $\ve>0$. Then $\bar{x}$ is a local minimizer of $\varphi$ if there exists a neighborhood $W$ of $(\bar{x},0)$ such that
\begin{equation}\label{eq:neighborlocalmin}
\langle z,w\rangle\ge 0\;\text{ whenever }\;z\in\widehat{D}^*T(u,v)(w),\; (u,v)\in\gph T \cap(U\times V),\;w\in\R^n,
\end{equation}
where $T$ is the $\varphi$-attentive $\epsilon$-localization of $\partial\varphi$ around $(\ox,0)$. If in addition that $\varphi$ is subdifferentially continuous at $\ox$ for $\ov$, then $\Hat{D}^* T(u,v)$ in \eqref{eq:neighborlocalmin} can be replaced by $\breve{\partial}^2\varphi(u,v)$.
\end{Corollary}
\begin{proof} This follows from Theorem~\ref{theo:codforvc} due to the easily verifiable fact that the variational convexity of $\ph$ at $\ox$ for $0\in\partial\ph(\ox)$ implies that $\ox$ is a local minimizer of $\ph$.
\end{proof} 

Note that $\ox$ may be a local minimizer of a function $\ph\colon\R^n\to\oR$ that is not variationally convex at $\ox$ for $0\in\partial\ph(\ox)$. An example of such a function $\ph\colon\R\to\R$, which is continuous on $\R$ but not prox-regular at $\ox$ for $0$, is presented in \cite[Remark~2.8(i)]{kmp22convex}. Furthermore, replacing in Corollary~\ref{cor:2nd-suff} the constructions $\Hat{D}^*T(u,v)$ and $\breve{\partial}^2\varphi(u,v)$ by $D^*T(u,v)$ and $\partial^2\ph(u,v)$, respectively, also gives us sufficient conditions for local minimizers of $\ph$ since the latter sets are larger than the former ones.
\vspace*{-0.12in}

\section{Second-Order Conditions for Strong Local Minimizers}\label{sec:strong}\vspace*{-0.05in}

In this section, we derive a counterpart of Theorem~\ref{PSDnecessary} as a second-order subdifferential {\em necessary} optimality conditions for {\em strong} local minimizers of prox-regular functions and provide further discussions of the {\em no-gap sufficiency} of such conditions and other relationships for this class of minimizers. 

Recall that $\ox$ is a {\em strong local minimizer} of an extended-real-valued function $\varphi:\R^n\to\oR$ with {\em modulus} $\sigma>0$ if $\ox\in\dom\varphi$ and  there exists a neighborhood $U$ of $\ox$ such that
$$
\varphi(x) \geq  \varphi(\ox) +\frac{\sigma}{2}\|x-\ox\|^2\;\text{ for all }\; x \in U.  
$$ 
Define further the {\em quadratically $\sigma$-shifted function} associated with $\varphi$  by
\begin{equation}\label{shift}
\psi(x):=\varphi(x) -\frac{\sigma}{2}\|x-\ox\|^2\;\text{ for all } \ x\in\R^n.
\end{equation}  
The following relationship between local and strong local minimizers is obvious.

\begin{Lemma}\label{lem:localminishift} Let $\varphi: \R^n\rightarrow\overline{\R}$ be an arbitrary extended-real-valued function, and let $\sigma>0$. Then  $\ox$ is a strong local minimizer of $\varphi$ with modulus $\sigma$ if and only if $\ox$ is a local minimizer of the quadratically $\sigma$-shifted function $\psi$ defined in \eqref{shift}.
\end{Lemma}

The above observation can be used to derive a necessary condition for strong local minimizers of $\ph$ in terms of the combined second-order subdifferentials from the one established in Theorem~\ref{NecessI} for standard local minimizers. To proceed in this direction, we need an appropriate sum rule for regular coderivatives, which incorporates the summation representation of $\ph$ in \eqref{shift} due to the combined second-order subdifferential structure in Definition~\ref{2ndsub}(ii). The new {\em regular coderivative sum rule} obtained below is of its own interest while being instrumental to achieve our goal.

\begin{Lemma}\label{sumrulecombine} Let $f:\R^n \to \R^m$,  let $F\colon\R^n\tto\R^m$, and let $\ov \in (f +F)(\ox)$ for some $\ox\in\R^n$. Assume that $f$ is {\em calm} at $\ox$ meaning that there exist $\ell >0$ and a neighborhood $U$ of $\ox$ such that 
\begin{equation}\label{calmF}
\|f(x) -f(\ox)\| \leq \ell \|x-\ox\|\;\text{ for all }\;x \in U. 
\end{equation}
Then we have the following inclusion:
\begin{equation}\label{reg-cod-sum}
\widehat{D}^*f (\ox)(w) + \widehat{D}^* F\big(\ox, \ov -f(\ox)\big)(w) \subset \widehat{D}^*(f+F)(\ox, \ov)(w)\;\mbox{ whenever }\;w \in \R^n. 
\end{equation}
If $f$ is Fr\'echet differentiable at $\ox$, then \eqref{reg-cod-sum}  holds as an equality. 
\end{Lemma}
\begin{proof} Observe first that we always have the inclusion 
\begin{equation}\label{frechetnormalinc}
\widehat{N}_{\Omega_1}(\oz) +\widehat{N}_{\Omega_2}(\oz) \subset \widehat{N}_{\Omega_1 \cap \Omega_2}(\oz)
\end{equation}
for any sets $\Omega_1, \Omega_2 \subset \R^s$ with $\oz \in \Omega_1\cap \Omega_2$. To verify \eqref{reg-cod-sum}, consider the sets
$$
\Omega_1:= \big\{(x,y_1,y_2)\in\R^n\times \R^m\times \R^m\;\big|\; (x,y_1) \in \gph f\big\},
$$
$$
\Omega_2:=\big\{(x,y_1,y_2)\in\R^n\times \R^m\times \R^m\;\big|\; (x,y_2) \in \gph F\big\}.
$$
Fixing $w\in \R^n$ and picking any $z \in \widehat{D}^*f (\ox)(w) + \widehat{D}^* F(\ox, \oy_2)(w)$ with $\oy_1:=f(\ox)$ and $\oy_2:= \ov - f(\ox)$, we find $z_1 \in \widehat{D}^*f(\ox)(w)$ and $z_2 \in \widehat{D}^*F(\ox,\oy_2)(w)$ such that $z=z_1+z_2$. This readily implies that
$$
(z_1, -w)\in \widehat{N}_{\text{\rm gph}f}\left( \ox, \oy_1 \right) \; \text{ and }\; (z_2, -w)\in \widehat{N}_{\text{\rm gph}F}\left(\ox, \oy_2 \right).
$$
Employing \cite[Proposition 1.2]{Mordukhovich06} gives us the inclusions
$$
(z_1,-w,0)\in \widehat{N}_{\Omega_1}(\ox,\oy_1, \oy_2) \; \text{and }\; (z_2,0,-w)\in \widehat{N}_{\Omega_2}(\ox,\oy_1, \oy_2),
$$
which ensure by \eqref{frechetnormalinc} that
\begin{equation}\label{normalo1o2}
(z_1+z_2, -w, -w) \in \widehat{N}_{\Omega_1\cap \Omega_2}(\ox,\oy_1, \oy_2).   
\end{equation}
To verify \eqref{reg-cod-sum}, we need to show that 
\begin{equation}\label{norofsum}
(z_1+z_2, -w) \in \widehat{N}_{\text{\rm gph} (f+F)}(\ox,\ov). 
\end{equation}
To proceed, pick any $\epsilon>0$ and deduce from \eqref{normalo1o2} that there exists $\delta>0$ such that $B_\delta(\ox) \subset U$, and that for all $(x,y_1,y_2)\in (\Omega_1\cap\Omega_2)\cap B_\delta((\ox,\oy_1, \oy_2))$ we have 
$$ 
\langle z_1 + z_2, x - \ox\rangle -\langle w, y_1 + y_2 -\ov    \rangle   \leq   \epsilon\big(\|x-\ox\|+\|y_1 - f(\ox)\| + \|y_2 - \ov + f(\ox)\|\big).
$$
Combining this with the calmness assumption in \eqref{calmF} tells us that
\begin{equation}\label{inner1}
\langle z_1 + z_2, x - \ox\rangle -\langle w, y_1 + y_2 -\ov    \rangle \leq M\epsilon(\|x -\ox\| + \|y_1+y_2 -\ov\|)
\end{equation}
for all $(x,y_1,y_2)\in (\Omega_1\cap\Omega_2)\cap B_\delta((\ox,\oy_1, \oy_2))$, where $M:=1+ 2\ell$. Choosing now
$$
\delta^\prime = \frac{\delta}{\sqrt{3(1+\ell^2)}},
$$
we readily arrive at the implication
\begin{equation}\label{grpimplies}
(x,y) \in \gph (f+F) \cap\big(B_{\delta^\prime}(\ox) \times B_{\delta^\prime}(\ov)\big) \Longrightarrow\big(x,f(x), y-f(x)\big) \in (\Omega_1\cap\Omega_2)\cap B_\delta((\ox,\oy_1, \oy_2)).
\end{equation}
Employing \eqref{inner1} together with \eqref{grpimplies} gives us the limiting condition 
$$
\underset{(x,y) \overset{ \text{\rm gph} (f+F) }{\longrightarrow} (\bar{x},\ov)} {\limsup} \frac{\langle z_1 + z_2, x - \ox\rangle -\langle w, y -\ov    \rangle }{\|x-\ox\|+\|y-\ov\|} \leq 0, 
$$
which yields \eqref{norofsum} by the definition of regular normals in \eqref{FrechetSubdifferential} and \eqref{NormalCones}. This means that $z \in \widehat{D}^*(f+F)(\ox,\ov)(w)$, and thus inclusion \eqref{reg-cod-sum} is justified. The equality in \eqref{reg-cod-sum} under the Fr\'echet differentiability of $f$ at $\ox$ follows from \cite[Theorem~1.62(i)]{Mordukhovich06}. 
\end{proof}

Now we are ready to derive the aforementioned {\em second-order necessary condition} for {\em strong} local minimizers of prox-regular functions  via the combined second-order subdifferential. 

\begin{Theorem}\label{NecessII} Let $\ox$ be a strong local minimizer with some modulus $\sigma>0$ of an extended-real-valued function $\varphi:\R^n\to\overline{\R}$, which is assumed to be prox-regular at $\ox$ for $0\in\partial\ph(\ox)$. Then the second-order subdifferential necessary optimality condition 
\begin{equation}\label{PSDnecessaryII}
\langle z,w \rangle \geq \sigma\|w\|^2\;\text{ for all }\;z \in \breve{\partial}^2\varphi(\bar{x},0)(w)\;\mbox{ and }\;w \in \R^n
\end{equation} 	  
is satisfied. This implies that we have, without a particular modulus specified, the following pointbased second-order necessary optimality condition $($PSONC$)$ for strong local minimizers:
\begin{equation}\label{PSDnecessaryIII}
\langle z,w\rangle >0 \;\text{ for all }\;z \in \breve{\partial}^2\varphi(\bar{x},0)(w)\;\mbox{ and }\;w\in\R^n\setminus\{0\}.
\end{equation}
\end{Theorem}
\begin{proof} Considering the quadratically shifted function $\psi$ from \eqref{shift}, we get by Lemma~\ref{lem:localminishift} that $\ox$ is  a local minimizer of $\psi$. Utilizing Theorem~\ref{NecessI} tells us that
\begin{equation}\label{2ndPSDforpsinec}
\langle z,w \rangle \geq 0\;\text{ whenever }\;z \in \breve{\partial}^2\psi(\bar{x},0)(w)\;\mbox{ and }\;w \in \R^n.
\end{equation} 	
It follows from the sum rule for regular coderivatives in Lemma~\ref{sumrulecombine} that 
\begin{equation}\label{sumruleshift}
\breve{\partial}^2\psi(\bar{x},0)(w) = \breve{\partial}^2\varphi(\bar{x},0)(w) -\sigma w\;\text{ for all }\;w\in \R^n.   
\end{equation}
Combining \eqref{2ndPSDforpsinec} and \eqref{sumruleshift}, we obtain \eqref{PSDnecessaryII} and thus complete the proof of the theorem. 
\end{proof}\vspace*{-0.1in}

\begin{Remark}[\bf comparison with known second-order necessary optimality conditions] \rm As advised by one of the referees, the usage of Moreau envelopes to obtain second-order optimality conditions was initiated by Eberhard and Wenczel \cite{eberhard}. In fact, our Theorems~\ref{NecessI} and \ref{NecessII} can be obtained via \cite[Theorem 67]{eberhard} since the proximal subdifferential and limiting subdifferential of $\varphi$ are the same under the prox-regularity of the function at the reference point, and since we may assume without loss of generality the prox-boundedness of a prox-regular function. 

However, our approach and proof, depending heavily on the relationships between local minimizers to a function and its corresponding Moreau envelope in Theorem~\ref{localMoreau}, explores deeply the celebrated notion of {\em prox-regularity} of extended real-valued functions together with its fruitful properties and interrelationships with the associated Moreau-Yosida regularization developed in \cite{Poliquin}. Further, our proof necessitates employing only the properties and calculus rules for regular coderivatives, while the approach in \cite{eberhard} requires using several auxiliary results involving other generalized second-order derivatives, such as the second-order lower Dini-directional derivative (or second-order subderivative) and the graphical derivative to the proximal subdifferential. Remarkably, it is observed that the proof in \cite{eberhard} cannot be easily simplified under the prox-regularity of the function in consideration. Below the main ideas in two approaches are provided:
\begin{itemize}
\item The assumed prox-regularity of $\varphi$ at $\ox$ for $0\in \partial\varphi (\ox)$ in Theorem \ref{NecessII} allows us, via Proposition~\ref{C11}, to exploit both conditions:
\begin{numcases}{}
\text{the local }\mathcal{C}^{1,1} \text{ smoothness of } e_{\lambda}\varphi \text{ around }\ox, \text{ and} \notag\\
\text{the fulfillment of } \nabla e_{\lambda}\varphi=\big(\lambda I + T^{-1}\big)^{-1} \text{ around }\ox \label{us-2},
\end{numcases}
where the second-order necessary optimality condition for functions of class $\mathcal{C}^{1,1}$, built in \cite[Theorem~3.3]{ChieuLeeYen17}, is applied later to the Moreau envelope $e_{\lambda}\varphi$ for obtaining the positive-definiteness of $(\widehat{D}^*\nabla e_{\lambda}\varphi) (\ox)$. At last, $(\widehat{D}^*\partial\varphi)(\ox,0)$ is positive-definite by the connection in \eqref{us-2}. 

\item On the other hand, the device in \cite[Theorem~67]{eberhard} employs,  under merely the prox-boundedness of $\varphi$, the automatic paraconcavity (i.e., weak concavity) of $e_{\lambda}\varphi$ to justify the following facts involving a strong local minimizer $\ox$:
\begin{numcases}{}
\text{The third-kind necessary condition from \cite[Definition~56]{eberhard} holds for $e_{\lambda}\varphi$ at }\ox,\text{ together with} \notag \\
\widehat{N}_{\text{gph} \partial_p \varphi}(\ox,0)\subset L_{\lambda}^T \Big(\widehat{N}_{\text{gph} \partial_p e_{\lambda}\varphi}(\ox,0)\Big), \text{ where }L_\lambda (x,y):=(x+\lambda y,y). \notag
\end{numcases}
Combining the obtained assertions, the authors  of \cite{eberhard} finally get the positive-semidefiniteness of tghe multifunction $(\widehat{D}^*\partial_p \varphi)(\ox,0)$.
\end{itemize}
\end{Remark}

\color{black}
As demonstrated in \cite{dl} by an illustrative example, a significant limitation of strong local minimizers lies in their high sensitivity with respect to minor perturbations. Addressing this issue, the authors of \cite{dl} introduce a more robust notion labeled {\em stable strong local minimizers}. Remarkably, it is proved in \cite{dl} that the subclass of strong local minimizers agrees with tilt-stable ones mentioned in Section~\ref{intro}.

Recall that $\ox\in\dom\varphi$ is a {\em tilt-stable local minimizer} of a function $\varphi\colon\R^n\to\oR$ if there exists a number $\gamma>0$ such that the mapping
\begin{equation*}
M_\gamma\colon v\mapsto{\rm argmin}\big\{\varphi(x)-\langle v,x\rangle\;\big|\;x\in\B_\gamma(\ox)\big\}
\end{equation*}
is single-valued and Lipschitz continuous on some neighborhood of $0\in\R^n$ with $M_\gamma(0)=\{\ox\}$. By a {\em modulus} of tilt stability of $\ph$ at $\ox$ we understand a Lipschitz constant of $M_\gamma$ around the origin. It has been well recognized starting with \cite{r19} that tilt-stable local minimizers of continuously prox-regular functions are closely related to strong variational stability of such functions with the moduli interplay. The following proposition is taken from \cite[Proposition~2.9]{kmp22convex}.

\begin{Proposition}\label{equitiltstr} Let $\varphi:\R^n\to\overline{\R}$ be continuously prox-regular at $\bar{x}\in\dom\varphi$ for $0\in\partial\varphi(\bar{x})$. Then we have the equivalent assertions:
		
{\bf(i)} $\varphi$ is strongly variationally convex at $\ox$ for $\ov=0$ with modulus $\sigma>0$. 
		
{\bf (ii)} $\ox$ is a tilt-stable local minimizer of $\varphi$ with modulus $\sigma^{-1}$. 
\end{Proposition}

Next we discuss some {\em second-order sufficient conditions} for strong local minimizers of continuously prox-regular functions that are complement the necessary optimality conditions from Theorem~\ref{NecessII}.

\begin{Remark}\label{SuffI} \rm Let $\varphi:\R^n\to \overline{\R}$ be continuously prox-regular at $\bar{x}$ for $0 \in \partial\varphi(\bar{x})$. One of the second-order subdifferential conditions to guarantee the {\em strong} local optimality of the stationary point $\ox$ of $\varphi$ is the {\em pointbased second-order sufficient optimality condition} (PSOSC) 
\begin{equation}\label{2ndPDtilt}
\langle z, w \rangle > 0\;\text{ whenever }\;z \in {\partial}^2\varphi(\bar{x},0)(w),\;w \in \R^n\setminus \{0\},
\end{equation}
which is different from \eqref{PSDnecessaryIII} only by the type of the second-order subdifferential. Indeed, it is shown in \cite[Theorem~1.3]{Poli} that PSOSC \eqref{2ndPDtilt} is {\em equivalent} to the {\em tilt-stability} of $\bar{x}$, which clearly implies that $\ox$ is a strong local minimizer of $\varphi$. Another second-order subdifferential condition ensuring the local optimality of $\ox$ for $\ph$ is the {\em neighborhood second-order sufficient optimality condition} (NSOSC)
\begin{equation}\label{2ndPSDaround}
\langle z, w\rangle \geq \sigma \|w\|^2\;\text{ whenever }\; z \in \breve{\partial}^2\varphi(x,v)(w), \; (x,v)\in\gph\partial\varphi\cap (U\times V), \; w \in \R^n
\end{equation}
valid for some $\sigma >0$ and neighborhoods $U$ of $\ox$ and $V$ of $\ov=0$ as a {\em no-gap neighborhood} counterpart of the second-order necessary optimality condition in \eqref{PSDnecessary}. Indeed, it is shown in \cite[Theorem~5.4]{kmp22convex} that NSOSC \eqref{2ndPSDaround} is a {\em characterization} of the {\em strong variational convexity} of a continuously prox-regular function $\varphi$ at $\ox$ for $0\in\partial\ph(\ox)$, which implies that $\ox$ is a strong (even tilt-stable) local minimizer of $\ph$ while not vice versa; see \cite[Remarks~2.6(ii)]{kmp22convex}. 
\end{Remark}\vspace*{-0.1in}

Now we draw the reader's attention to Figure~\ref{fig:my_label} illustrating the relationships between the obtained second-order subdifferential optimality conditions for strong local minimizers and the associated variational notions. In this figure, the arrows $(\longrightarrow)$, which are based on Theorem~\ref{NecessII}, Proposition~\ref{equitiltstr}, and Remark~\ref{SuffI}, read ``implies''. On the other hand, the arrows $(\dashrightarrow)$ read ``may not imply''. In addition to the previous discussions on ($\dashrightarrow$), observe that we do not have the equivalence between strong local minimizers and the pointbased second-order necessary optimality condition (PSONC) \eqref{PSDnecessaryIII}.  Indeed, it is shown in \cite[Example 4.10]{ChieuLeeYen17} that there exists a function of class $\mathcal{C}^{1,1}$ (hence continuously prox-regular) for which  PSONC \eqref{PSDnecessaryIII} holds but $\ox$ is not even a local minimizer of $\varphi$.
\vspace*{-0.3in}
\begin{center} 
\begin{figure}
\centering
\begin{tikzpicture}[x=0.75pt,y=0.75pt,yscale=-1,xscale=1, line width=0.75pt]
\draw  [->] (153,254) -- (227,254) ;
\draw  [  ->]  (227,247) -- (153,247) ;
\draw  [->]  (285,254) -- (359,254) ;
\draw   [dashed, thick, ->] (359,247) -- (285,247) ;
\draw  [->]  (139,238) -- (177,172) ;
\draw  [->]  (167,172.1) -- (130.05,236.75);
\draw [->] (337,174) -- (375,240);
\draw  [dashed, ->] (384,237) -- (345,170);
\draw  [ - >]   (219,81) -- (181,147) ;
\draw  [->]  (191,147) -- (229,81) ;
\draw [->]   (285,81) -- (323,147);
\draw  [-> , dashed]  (334,147) -- (296,81) ;
\draw  [->]   (252,235) -- (252,81)  ;
\draw  [->]  (260,81) --  (260,235) ;
\draw (100,244) node [anchor=north west][inner sep=0.75pt]    {PSOSC};
\draw (232,244) node [anchor=north west][inner sep=0.75pt]    {NSOSC};
\draw (365,244) node [anchor=north west][inner sep=0.75pt]    {PSONC};
\draw (128,154) node [anchor=north west][inner sep=0.75pt]    {Tilt Stability};
\draw (298,154) node [anchor=north west][inner sep=0.75pt]    {Strong Local Minimizer};
\draw (189,60) node [anchor=north west][inner sep=0.75pt]    {Strong Variational Convexity};
\end{tikzpicture}
\caption{Relationships between second-order optimality conditions, tilt stability, strong variational convexity, and strong local minimizer for continuously prox-regular functions}
\label{fig:my_label}
\end{figure}
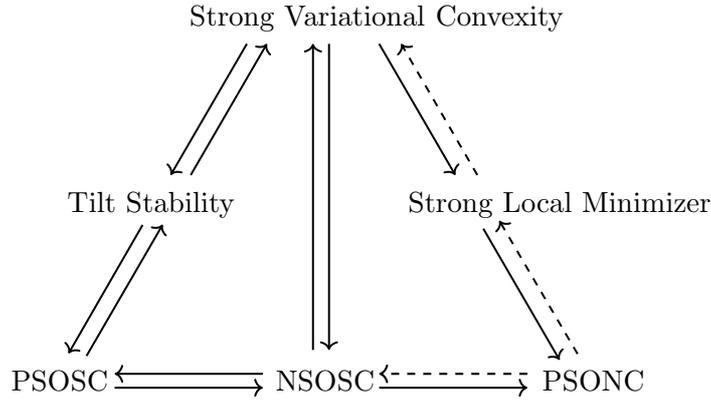
\end{center}
\vspace*{-0.2in}

\section{Second-Order Necessary Conditions in Constrained Optimization}\label{sec:optiforconstrained}\vspace*{-0.05in}

In this section, we aim at establishing second-order necessary  optimality conditions for the class of {\em constrained optimization problems} in the functional-geometric form given by:
\begin{equation}\label{composite}
\min\;f(x)\;\text{ subject to }\;g(x)\in\Theta,
\end{equation}
where $f:\R^n\to\R$ and $g\colon\R^n\to\R^m$ are ${\cal C}^1$-smooth at the reference points, and where $\Theta \subset \R^m$ is a nonempty, closed, and convex set. These are our {\em standing assumptions} in this section unless otherwise stated. Problem \eqref{composite} encompasses various particular classes of problems in constrained optimization such as conic programs, nonlinear programs, etc.

Denoting the set of feasible solutions to \eqref{composite} by 
\begin{equation}\label{Gammas}
\Gamma:=\big\{x\in\R^n\;\big|\;g(x) \in \Theta\big\}, 
\end{equation}
we can rewrite \eqref{composite} in the unconstrained extended-real-valued format 
\begin{equation}\label{composiuncon}
\min \quad \varphi(x):= f(x) + \delta_{\Gamma}(x)\;\text{ subject to }\; x \in \R^n 
\end{equation}
for which the second-order optimality conditions are established in Sections~\ref{sec:optiforprox}--\ref{sec:strong} via the second-order subdifferentials of the cost function $\ph\colon\R^n\to\oR$. Our goal is to obtain second-order optimality conditions in terms of the initial data of problem \eqref{composite}. This section addresses necessary optimality conditions, while the corresponding sufficient conditions for \eqref{composite} are derived in Section~\ref{sec:2ndsuf}.

It follows from \cite[Proposition~1.30] {Mor18} that the {\rm first-order necessary condition} for the local optimality of $\ox$ in the unconstrained optimization problem \eqref{composiuncon} is 
\begin{equation}\label{Fermatrulecompo}
0 \in \partial\varphi(\ox) = \nabla f(\ox) + \partial\delta_\Gamma(\ox)= \nabla f(\ox) + N_\Gamma(\ox),
\end{equation}
which can be equivalently expressed in the form
\begin{equation}\label{1stcond}
\nabla f(\ox) + \nabla g(\ox)^*\oy = 0 \; \text{ for some }\;\oy \in N_{\Theta}(g(\ox))
\end{equation}
under appropriate constraint qualification conditions. One of the most well-known constraint qualification conditions, which is called the {\rm first-order constrained qualification condition} (FOCQ), to guarantee the equivalence between \eqref{Fermatrulecompo} and \eqref{1stcond} is 
\begin{equation}\label{1stqualifyconcompos}
N_\Theta (g(\ox)) \cap \ker \nabla g(\ox)^* = \{0\}. 
\end{equation}
For particular classes of optimization problems, the general FOCQ \eqref{1stqualifyconcompos}
reduces to the classical {\em Mangasarian-Fromovitz constraint qualification} (MFCQ) in nonlinear programming and to the {\em Robinson constraint qualification} in conic programming. Moreover, it follows from the Mordukhovich criterion in \eqref{cor-cr} (see \cite[Theorem~3.6]{Mordu93} and \cite[Theorem~9.40]{Rockafellar98}) that condition \eqref{1stqualifyconcompos} is equivalent to the \textit{metric regularity} of the   mapping $(x,\alpha)\mapsto (g(x),\alpha) -\Theta \times \R_{+}$ around the point $((\ox,0),(0,0))$. We are also going to use the qualification condition labeled as the {\em second-order qualification condition} (SOCQ) from \cite{mor-roc} for problem \eqref{NLPproblem} at $\ox \in \Gamma$, which is formulated as follows:
\begin{equation}\label{SOCQ}
D^*N_\Theta(g(\ox), \oy)(0) \cap \ker \nabla g(\ox)^* =\{0\},  
\end{equation}
where $\oy \in \R^m$ satisfies \eqref{1stcond}. It follows from \cite[Lemma~7.2]{kmp22convex} that SOCQ \eqref{SOCQ} reduces to the classical {\em linear independence constraint
qualification} (LICQ) in nonlinear programming. Next we recall the weaker qualification condition, which is labeled as the {\em metric subregularity constraint qualification condition} (MSCQ) for \eqref{composite}, introduced in \cite{gm15}. Namely, we say that MSCQ holds at $\ox \in \Gamma$ if the mapping $x \mapsto g(x) - \Theta$ is {\em metrically subregular} at $(\ox,0)$, i.e., 
there exist a neighborhood $U$ of $\ox$ and a number $\kappa> 0$ such that the distance estimate 
\begin{equation}\label{MSCQ}
\dist(x;\Gamma) \leq \kappa\,\dist (g(x);\Theta) \; \text{ for all }\; x\in U
\end{equation}
is satisfied. It follows from the definition that MSCQ is a {\em robust} property, i.e., MSCQ is fulfilled at every  $x \in \Gamma$ near $\ox$ whenever it holds
at $\ox \in \Gamma$. The next remark summarizes relationships between the above constraint qualification conditions that are broadly used in this paper.

\begin{Remark}\label{CQdiscuss} \rm Observe the following:
 
{\bf(i)} We get from \cite[Proposition~6.3]{kmp22convex} that SOCQ always yields FOCQ while the inverse implication is not true in general. This distinction is evident in classical nonlinear programming, where FOCQ is equivalent to MFCQ while SOCQ reduces to LICQ, which is stronger than MFCQ.

{\bf(ii)} As discussed above, FOCQ is equivalent to the metric regularity of the mapping $(x,\alpha)\mapsto (g(x),\alpha) -\Theta \times \R_{+}$ around the point $((\ox,0),(0,0))$. Using \cite[Proposition~3.1]{mms}, we deduce that the latter property implies that $g(x) -\Theta$ is metrically subregular at $(\ox,0)$, i.e., MSCQ is satisfied at $\ox$. The inverse implication fails in general. This is demonstrated by
\cite[Example~3.3(b)]{mms}.
\end{Remark}

The relationships between FOCQ, SOCQ, and MSCQ are illustrated in Figure~\ref{fig:CQ},
where the arrows $(\longrightarrow)$ read ``implies'' while the arrows $(\dashrightarrow)$ read ``may not imply''.\vspace*{-0.3in}
\begin{center}
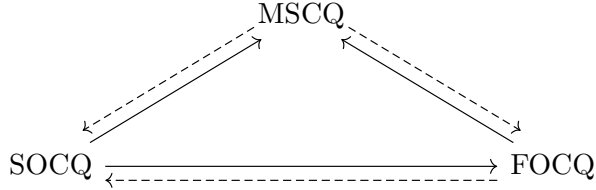
\begin{figure}
\centering
\begin{tikzcd}
&& \text{MSCQ} \\
\\
\text{SOCQ} &&&& \text{FOCQ}
\arrow[from=3-1, to=3-5]
\arrow[from=3-5, to=1-3]
\arrow[shift left=-2, dashed, from=1-3, to=3-1]
\arrow[shift left=2, dashed, from=1-3, to=3-5]
\arrow[shift left=2, dashed, from=3-5, to=3-1]
\arrow[from=3-1, to=1-3]
\end{tikzcd}
\caption{Relationships between constraint qualification conditions}
\label{fig:CQ}
\end{figure}
\end{center}

Next we recall a useful formula for calculating the regular coderivative of the normal cone mapping obtained in \cite[Corollary~3.8]{chieuhien} under MSCQ. 

\begin{Proposition}\label{chainrulecombine} Let  $\Theta \subset \R^m$ be a  nonempty polyhedral convex set, let the set $\Gamma \subset \R^n$ be given in \eqref{Gammas} where $g: \R^n\to  \R^m$ is a $\mathcal{C}^2$-smooth mapping. Suppose that MSCQ is satisfied at $\ox \in \Gamma$ and for each  $\ov \in N_\Gamma(\ox)$ define the sets
$$
P(\ox,\ov):=\big\{y\in\R^m\;\big|\;\ov=\nabla g(\ox)^*y,\;y\in N_{\Theta}\left(g(\ox)\right) \big\},
$$
\begin{equation}\label{linearprogramset}
S(v):= \text{\rm argmax} \left\{\langle v,\nabla^2 \langle y, g\rangle (\ox)v\rangle\;\big|\; y \in P(\ox,\ov)  \right\}  \; \text{for all } v \in K_\Gamma(\ox,\ov).  
\end{equation}
Then we have the calculation formula
\begin{equation*}
\widehat{D}^*N_{\Gamma}(\ox,\ov)(w) = \begin{cases}
\big\{u\;\big|\; \langle u, v\rangle -\langle w, \nabla^2 \langle y, g\rangle (\ox)v\rangle   \leq 0\;  \forall v \in K_\Gamma(\ox,\ov), \; y \in S(v)\big\} & \text{if }\;w \in - K_\Gamma(\ox,\ov),\\
\emptyset & \text{otherwise},
\end{cases}
\end{equation*}
where $K_\Gamma$ stands for the critical cone \eqref{criticalconeK}
associated with the set of feasible solutions \eqref{Gammas}.
\end{Proposition}

Employing the above coderivative calculation together with other results of variational analysis leads us to the following novel {\em second-order necessary optimality conditions} for local and strong local minimizers of the constrained problem \eqref{composite} involving the initial data of this problem. 

\begin{Theorem}\label{necforconstrained} Consider the constrained optimization problem \eqref{composite} in which $\Theta$ is a nonempty convex polyhedron. Given a local minimizer $\ox$ of \eqref{composite}, assume that  $f$ is of class $\mathcal{C}^{1,1}$ and $g$ is $\mathcal{C}^2$-smooth around $\ox$, and that MSCQ holds at this point. Then $\ox$ satisfies the first-order optimality
condition \eqref{1stcond}. If in addition $\oy$ in \eqref{1stcond} is unique, then we have the second-order necessary optimality condition
\begin{equation}\label{necescompcondi}
\langle z, w\rangle +  \langle w, \nabla^2 \langle \oy, g\rangle (\ox)w\rangle \geq 0
\end{equation}
for all $w \in -K_\Gamma(\ox,-\nabla f(\ox))$ and $z \in \breve{\partial}^2 f(\ox)(w)$.  In the case where $\ox$ is a strong local minimizer of \eqref{composite} with modulus $\sigma>0$, the second-order optimality condition
\begin{equation}\label{necescompcondiII}
\langle z, w\rangle +  \langle w, \nabla^2 \langle \oy, g\rangle (\ox)w\rangle \geq \sigma\|w\|^2
\end{equation}
holds for all $w \in -K_\Gamma(\ox,-\nabla f(\ox))$ and $z \in \breve{\partial}^2 f(\ox)(w)$. 
\end{Theorem}
\begin{proof}
By combining \cite[Proposition~1.30]{Mor18} and \cite[Lemma~2.5]{chieuhien} with the assumed local optimality at $\ox$, we get the equalities
\begin{equation*}
0\in \partial \varphi (\ox) = \nabla f(\ox) + N_{\Gamma}(\ox) = \nabla f(\ox) + \nabla g(\ox)^* N_{\Theta}\big( g(\ox)\big),
\end{equation*}
which give us the first-order optimality condition \eqref{1stcond}. The MSCQ at $\ox$ and the polyhedral convexity of $\Theta$ yield the continuous prox-regularity of $\delta_\Gamma$ at $\ox$ due to \cite[Theorem~7.1]{mms}. Using this together with the $\mathcal{C}^{1,1}$ property of $f$ around $\ox$, we get that the function $\varphi=f+\delta_\Gamma$ is continuously prox-regular at $\ox$ for $0$ due to \cite[Exercise~13.35]{Rockafellar98}. 

Next we verify the second-order condition \eqref{necescompcondi}   under the additional assumption that $\oy$ in \eqref{1stcond} is unique. Employing Theorem~\ref{NecessI} tells us that if $\ox$ is a local minimizer of $\varphi$, then 
\begin{equation}\label{PSDnecessarycomposite}
\langle z, w \rangle \geq 0\; \text{ whenever }\;z \in \breve{\partial}^2\varphi(\bar{x},0)(w),\;w \in \R^n.
\end{equation}
Furthermore, the coderivative sum rule from Lemma~\ref{sumrulecombine} ensures that
\begin{equation}\label{sumcombine}
\breve{\partial}^2 f(\ox)(w) + \breve{\partial}^2 \delta_\Gamma (\ox,-\nabla f(\ox))(w) \subset \breve{\partial}^2 \varphi (\ox,0)(w)\;\text{ for all }\; w \in \R^n. 
\end{equation}
Since $-\nabla f(\ox)\in N_\Gamma(\ox)$ and $\oy$ in \eqref{1stcond} is unique, it follows from Proposition~\ref{chainrulecombine} that 
$$
\widehat{D}^*N_{\Gamma}(\ox,-\nabla f(\ox))(w)=\big\{u\;\big|\; \langle u, v\rangle -\langle w, \nabla^2 \langle \oy, g\rangle (\ox)v\rangle   \leq 0\;  \forall v \in K_\Gamma(\ox,-\nabla f(\ox))\big\}
$$
for any $w\in -K_\Gamma(\ox,-\nabla f(\ox))$, and therefore
\begin{equation}\label{chaincombn}
\nabla^2\langle \oy, g\rangle (\ox)w   \in  \widehat{D}^*N_{\Gamma}(\ox,-\nabla f(\ox))(w) = \breve{\partial}^2\delta_\Gamma (\ox,-\nabla f(\ox))(w).
\end{equation}
Combining \eqref{sumcombine} and \eqref{chaincombn} brings us to the inclusion
\begin{equation}\label{inclnecess}
z + \nabla^2 \langle \oy, g\rangle (\ox)w \in \breve{\partial}^2\varphi(\ox,0)(w)\;\mbox{ for all }\;w\in -K_\Gamma(\ox,-\nabla f(\ox))\;\mbox{ and }\;z \in \breve{\partial}^2 f(\ox)(w).
\end{equation}
Using the latter together with \eqref{PSDnecessarycomposite} yields
$$
\langle z, w\rangle + \langle w,\nabla^2 \langle \oy, g\rangle (\ox) w\rangle \geq 0\;\mbox{ for all }\;w\in -K_\Gamma(\ox,-\nabla f(\ox))\;\mbox{ and }\;z \in \breve{\partial}^2 f(\ox)(w),
$$
which verifies \eqref{necescompcondi}. If $\ox$ is a strong local minimizer of \eqref{composite} with modulus $\sigma>0$, then we have by Theorem~\ref{NecessII} that
\begin{equation}\label{PSDnecessarycompositeII}
\langle z, w \rangle \geq \sigma\|w\|^2\;\text{ whenever }\;z \in \breve{\partial}^2\varphi(\bar{x},0)(w),\;w \in \R^n,
\end{equation} 	
and thus deduce from \eqref{PSDnecessarycompositeII} and \eqref{inclnecess} the second-order estimate 
$$
\langle z, w\rangle + \langle w,\nabla^2 \langle \oy, g\rangle (\ox) w\rangle \geq \sigma\|w\|^2\;\mbox{ for all }\;w\in -K_\Gamma(\ox,-\nabla f(\ox))\;\mbox{ and }\;z \in \breve{\partial}^2 f(\ox)(w).
$$
This justifies \eqref{necescompcondiII} and completes the proof of the theorem. 
\end{proof}

If MSCQ is replaced by SOCQ \eqref{SOCQ} in Theorem~\ref{necforconstrained}, then the required uniqueness of $\oy$ is automatically satisfied, and thus we arrive at: 

\begin{Corollary} \label{necforconstrainedSOCQ} Consider the optimization problem \eqref{composite}, where $\Theta$ is a nonempty convex polyhedron, where $f$ and $g$ are of class $\mathcal{C}^{1,1}$ and $\mathcal{C}^2$ around $\ox$, respectively, and where SOCQ \eqref{SOCQ} is satisfied at $\ox$. Then the following assertions hold:

{\bf(i)} If $\ox$ is a local minimizer of  problem \eqref{composite}, then $\ox$ satisfies the first-order optimality condition \eqref{1stcond}, where $\oy$ is unique. Moreover,  the second-order necessary condition \eqref{necescompcondi} is satisfied. 

{\bf(ii)} If $\ox$ is a strong local minimizer of \eqref{composite} with modulus $\sigma>0$, then $\ox$ satisfies the first-order optimality condition \eqref{1stcond}, where $\oy$ is unique, and the second-order condition \eqref{necescompcondiII} is satisfied.
\end{Corollary}
\begin{proof} By Remark~\ref{CQdiscuss}, MSCQ holds at $\ox$. It follows from \cite[Theorem~4.3]{mor-roc} that $\oy$ in \eqref{1stcond} is unique under SOCQ. Therefore, both (i) and (ii) are implied by Theorem~\ref{necforconstrained}. 
\end{proof}

Note that SOCQ \eqref{SOCQ} is merely a sufficient condition to ensure MSCQ \eqref{MSCQ} and the uniqueness of $\oy$ in \eqref{1stcond}. The following example presents an optimization problem of type \eqref{composite} where MSCQ is satisfied and $\oy$ is unique in \eqref{1stcond}, while SOCQ fails.

\begin{Example}\label{nlpex} \rm Consider problem \eqref{composite} with the initial data
$$
f(x_1,x_2):= x_1^2 - x_2^2, \; g(x_1,x_2):=(-x_1+x_2, -x_1-x_2, x_2), \;
\Theta:= \R_{-}\times \R_{-} \times \{0\}
$$
for all $(x_1,x_2) \in \R^2$. Hoffman’s lemma \cite[Theorem~2.200]{bonsha} tells us that the mapping $x \mapsto g(x)-\Theta $ is metrically subregular at any point in $\R^2$, and so MSCQ is satisfied at any point of the plane. It is easy to see that
$N_\Theta(g(\ox)) = \R_{+}\times \R_{+}\times \R$ for $\ox:=(0,0)$, and thus condition \eqref{1stcond} reduces to
\begin{equation*}
\nabla f(\ox)+\nabla g(\ox)^*\oy=0\;\mbox{ for some }\;\oy=(\oy_1, \oy_2, \oy_3)\in \R_{+}\times \R_{+}\times \R,\;\mbox{ i.e., }
\end{equation*}
$$
-\oy_1 -\oy_2 =0, \; \oy_1 - \oy_2 + \oy_3 =0 \; \text{ for some }\; \oy=(\oy_1, \oy_2, \oy_3)\in \R_{+}\times \R_{+}\times \R. 
$$
The latter gives us $\oy_1=\oy_2=\oy_3 =0$, and hence $\oy = (0,0,0)$ is the unique vector satisfying \eqref{1stcond}. On the other hand, we deduce from \cite[Lemma~7.1]{kmp22convex} that  
$$
\begin{aligned}
D^*N_\Theta(g(\ox), \oy)(0)\cap \text{\rm ker} \nabla g(\ox)^* &=  \partial^2 \delta_{\Theta}(g(\ox))(0)\cap \text{\rm ker} \nabla g(\ox)^* \\
&=\big\{u=(u_1,u_2,u_3)\in\R^3\;\big|\;-u_1 -u_2 =0, u_1 -u_2 +u_3 =0\big\},
\end{aligned}
$$
which implies that the set $D^*N_\Theta(g(\ox), \oy)(0)\cap \text{\rm ker} \nabla g(\ox)^*$ is not a singleton, and therefore
the SOCQ condition \eqref{SOCQ} is not satisfied. 
\end{Example}

Finally in this section, we discuss relationships of our new second-order subdifferential necessary optimality conditions with known results in this direction.

\begin{Remark}\label{rem:2nc}
\rm Second-order necessary conditions for constrained optimization problems of type \eqref{composite} were obtained in the recent papers \cite{anyenxu23,anyen21} via the regular second-order subdifferential, which coincides with the combined one for $\mathcal{C}^1$-smooth functions; see \eqref{Frechet_Mordukhovich}. Observe the following:

{\bf (i)} An and Yen \cite{anyen21} derived the second-order necessary condition for \eqref{composite} in the case where the  constraint set is a convex polyhedron, and where $f$ is $\mathcal{C}^1$-smooth such that its gradient calm; see \cite[Theorem~5]{anyen21}. Note that the set $\Gamma$ in \eqref{Gammas} in our case is not necessarily polyhedral, or even convex.  

{\bf(ii)} In \cite[Theorem 3.5]{anyenxu23}, the authors established the second-order necessary condition for \eqref{composite} with $\Theta={0_{\R^m}}$ and the surjective Jacobian $\nabla g(\ox)$. In contrast, our results in Theorem~\ref{necforconstrained} are obtained under much less restrictive constraint qualifications, and the set $\Theta$ under consideration can be any convex polyhedron. Note that the results of \cite[Theorems~3.10 and 3.11]{anyenxu23} provide alternative second-order necessary conditions for \eqref{composite} under the Robinson constraint qualification. However, the latter results are applicable only when the mapping $g$ in \eqref{composite} is affine, which is not required in our conditions.
\end{Remark} \vspace*{-0.2in}

\section{Second-Order Sufficient Conditions in Constrained Optimization}\label{sec:2ndsuf}\vspace*{-0.05in}

In this section, we derive second-order subdifferential sufficient conditions for local and strong local minimizers of problem \eqref{composite} under the mild MSCQ condition in the case when $\Th$ is convex. For each $(x,v)\in\R^n\times\R^n$, define the {\em set of multipliers} for \eqref{composite} as follows: 
\begin{equation}\label{Lagrangemultiplier}
\Lambda (x,v):= \big\{y\in\R^m\;\big|\;v=\nabla f(x)+\nabla g(x)^*y,\;y\in N_{\Theta}\left(g(x) \right) \big\}.
\end{equation}

The next proposition shows that the multiplier 
 mapping in \eqref{Lagrangemultiplier} possesses the local boundedness property. Recall that a set-valued mapping $F:\R^n \rightrightarrows \R^m$ is {\em locally bounded} around $x\in \dom F$ if there is a neighborhood $U$ of $x$ for which $F(U)$ is bounded in $\R^m$. 

\begin{Proposition}\label{local-closed}
Consider problem \eqref{composite}, where FOCQ \eqref{1stqualifyconcompos} holds. Then the set-valued mapping of Lagrange multipliers $\Lambda:\R^n \times \R^n \rightrightarrows \R^m$ defined in \eqref{Lagrangemultiplier}  is locally bounded around $(\ox,0)$.
\end{Proposition}
\begin{proof}
Assuming the contrary, construct a sequence $(x_k,v_k,y_k)\in \gph \Lambda$ for which $(x_k,v_k) \rightarrow (\ox,0)$ while $\|y_k\|\rightarrow \infty$ as $k\rightarrow \infty$. By the definition of $\Lambda(\cdot,\cdot)$, for every $k\in\N$ we have $y_k\in N_{\Theta}(g(x_k))$ and the relationship $v_k = \nabla f(x_k) + \nabla g(x_k)^* y_k$. This gives us therefore that
\begin{equation}\label{nvc}
\dfrac{v_k}{\|y_k\|}= \dfrac{\nabla f(x_k)}{\|y_k\|}+\nabla g(x_k)^*\left(\dfrac{y_k}{\|y_k\|}\right)\;\text{ for large }\;k\in \N.
\end{equation}
Suppose without loss of generality the convergence $\frac{y_k}{\|y_k\|}\rightarrow \tilde{y}$ as $k\to\infty$ for some $\tilde{y}\in \R^m$ with $\|\tilde{y}\|=1$. Moreover, since $(x_k,v_k)\rightarrow (\ox,0)$ and $\|y_k\|\rightarrow \infty$ while $f,g$ are continuously differentiable, the passage to the limit in \eqref{nvc} as $k\rightarrow \infty$ leads us to the equality
\begin{equation}\label{break-1}
\nabla g(\ox)^* \tilde{y}=0.  
\end{equation}
On the other hand, by $\frac{y_k}{\|y_k\|}\in N_{\Theta}(g(x_k))$ for large $k\in \N$ and the closedness of the set $\gph N_\Th$, we get $\tilde{y}\in N_{\Theta}(g(\ox))$. Combining the latter with \eqref{break-1} and employing the imposed FOCQ amount to $\tilde{y}=0$, which clearly contradicts the fact that $\|\tilde{y}\|=1$ and thus completes the proof.
\end{proof}

One more lemma is needed to establish the desired second-order sufficient condition below. 

\begin{Lemma}\label{nonemptylinear} In the setting of Proposition~{\rm\ref{chainrulecombine}}, we have that  $S(v)\ne \emptyset$ for all $v \in K_\Gamma(\ox,\ov)$, where the sets $S(v)$ are defined in \eqref{linearprogramset}. 
\end{Lemma}
\begin{proof} Since $\Theta$ is a convex polyhedron, we find $a_i \in \R^m$ and $b_i \in \R$ for $i = 1,\ldots, p$ such that 
$$
\Theta=\big\{y \in \R^m\;\big|\;\langle a_i, y\rangle \leq b_i, \; i=1,\ldots,p\big\}.
$$
Define $q(x):= A^*g(x) - b$, $x \in \R^n$, where $A$ is the $m\times p$ matrix with the $i$th column $a_i$, $i=1,\ldots,p$, and where $b=(b_1,\ldots, b_p)$. We rewrite $\Gamma$ as 
$\Gamma = \{x \in \R^n\;|\; q(x) \in \R^p_{-}\}$
and define the following sets:
$$
\widehat{P}(\ox,\ov):=\big\{z\in\R^p\;\big|\;\ov=\nabla q(\ox)^*z,\;z\in N_{\R_{-}^p}\left(q(\ox)\right)\big\},
$$
$$
\widehat{S}(v):= \text{\rm argmax} \left\{\langle v,\nabla^2 \langle z, q\rangle (\ox)v\rangle\;\big|\; z \in \widehat{P}(\ox,\ov)  \right\}  \; \text{ for all }\;v \in K_\Gamma(\ox,\ov).  
$$
Hoffman's lemma ensures that MSCQ is satisfied in our setting, i.e., the mapping $x\mapsto q(x)-\R_{-}^p$ is metrically subregular at $(\ox,0)$. Using the result of \cite[Proposition~4.3]{gm15} tells us that $\widehat{S}(v)\ne \emptyset$ for all $v \in K_\Gamma(\ox,\ov)$. Furthermore, arguing as in the proof of 
\cite[Theorem~3.5]{chieuhien} brings us to the equalities
$$
P(\ox,\ov) = A\big(\widehat{P}(\ox,\ov)\big)\; 
\text{ and }\; 
\langle v, \nabla^2 \langle z, q\rangle (\ox)v \rangle =  \langle v, \nabla^2 \langle Az, g\rangle (\ox)v\rangle \; \text{for all }\; z \in   \widehat{P}(\ox,\ov),
$$
which imply that  $S(v)= A(\widehat{S}(v))$ for all $v \in K_\Gamma(\ox,\ov)$. 
Combining this with the fact that  $\widehat{S}(v)\ne \emptyset$, we arrive at $S(v) \ne \emptyset$ for all $v \in K_\Gamma(\ox,\ov)$, which completes the proof of the lemma. 
\end{proof}

Here are the aforementioned {\em second-order subdifferential sufficient conditions} for both local and strong local minimizers in the constrained problem \eqref{composite}.

\begin{Theorem}\label{suff2ndconstraint} Consider problem \eqref{composite}, where $\Theta$ is a convex polyhedron, and where $\ox$ satisfies the first-order optimality condition \eqref{1stcond}. Assume that both $f$ and $g$ are $\mathcal{C}^{2}$-smooth around $\ox$, and that
MSCQ is fulfilled at this point. Then we have the following assertions:

{\bf(i)} If there exist neighborhoods $U$ of $\ox$ and $V$ of $0$ such that 
\begin{equation}\label{sufficient2nd}
\langle \nabla^2 f(x)w,w\rangle + \langle w, \nabla^2 \langle y, g\rangle (x)w\rangle \geq 0 
\end{equation}
for all $x \in U$,  $v \in V$, $w \in -K_\Gamma(x,v-\nabla f(x))$, and $y \in \Lambda(x,v )$, where the set-valued mapping $\Lambda$ is defined in \eqref{Lagrangemultiplier}, then the function $f+ \delta_{\Gamma}$ is variationally convex at $\ox$ for $\ov=0$.  Consequently, $\ox$ is a local minimizer of the constrained problem \eqref{composite}.

{\bf(ii)} If there exist neighborhoods $U$ of $\ox$ and $V$ of $0$ such that 
\begin{equation}\label{sufficient2ndII}
\langle \nabla^2 f(x)w,w\rangle + \langle w, \nabla^2 \langle y, g\rangle (x)w\rangle \geq \sigma\|w\|^2 
\end{equation}
for all $x \in U$,  $v \in V$, $w \in -K_\Gamma(x,v-\nabla f(x)),   \;  y \in \Lambda(x,v )$ and some $\sigma>0$, where the set-valued mapping $\Lambda$ is defined in \eqref{Lagrangemultiplier}, then the function $f+ \delta_{\Gamma}$ is strongly variationally convex at $\ox$ for $0$ with modulus $\sigma$. Consequently, $\ox$ is a strong (in fact tilt-stable) local minimizer of \eqref{composite} with modulus $\sigma^{-1}$.
\end{Theorem}
\begin{proof} It follows from \cite[Proposition~1.30] {Mor18} and \cite[Lemma~2.5]{chieuhien} that the first-order optimality condition \eqref{1stcond} is equivalent to the inclusion $0 \in \partial\varphi(\ox)$ with $\varphi:= f+ \delta_{\Gamma}$. Employing the same arguments in the proof of Theorem~\ref{necforconstrained} verifies that $\varphi$ is continuously prox-regular at $\ox$ for $0$. Now we are going to apply the results of \cite[Theorems~5.2 and 5.4]{kmp22convex} to the above function $\ph$ by showing that our assumptions guarantee the fulfillment of the conditions for the variational and strong variational convexity therein. To proceed, let us verify that there are neighborhoods $U$ of $\ox$ and $V$ of $0$ such that   
\begin{equation}\label{PSDsufcomposite}
\langle z, w \rangle \geq 0\;\text{ whenever }\;z \in \breve{\partial}^2\varphi(x,v)(w), \; (x,v)\in\gph\partial\varphi \cap (U\times V),\; w \in \R^n,
\end{equation} 	
\begin{equation}\label{PSDsufcompositeII}
\langle z, w \rangle \geq \sigma\|w\|^2\;\text{ whenever }\;z \in \breve{\partial}^2\varphi(x,v)(w), \; (x,v)\in\gph\partial\varphi \cap (U\times V),\; w \in \R^n,
\end{equation} 	
under the fulfillment of \eqref{sufficient2nd} and \eqref{sufficient2ndII}, respectively. 
Since MSCQ is satisfied at $\ox$, it also holds around this point. Assuming that \eqref{sufficient2nd} is fulfilled, we get without loss of generality that MSCQ is satisfied at any $x \in U$ and that $f$ is $\mathcal{C}^2$-smooth on $U$. It follows from  \cite[Theorem~1.62]{Mordukhovich06} that  
\begin{equation}\label{sumcombine2}
\breve{\partial}^2 \varphi (x,v)(w)=\nabla^2 f(x)w + \breve{\partial}^2 \delta_\Gamma (x,v-\nabla f(x))(w)\; \text{ for all }\;x \in U, \;  (x,v)\in\gph\partial\varphi,\; w \in \R^n. 
\end{equation} 
Hence whenever $x\in U$, $w \in \R^n$, $(x,v)\in \gph\partial\varphi$, and $z \in \breve{\partial}^2 \varphi(x,v)(w)$, we deduce from \eqref{sumcombine2} that 
$$
z - \nabla^2 f(x) w \in \breve{\partial}^2 \delta_\Gamma (x,v-\nabla f(x))(w)=\widehat{D}^*N_\Gamma(x,v-\nabla f(x))(w).
$$
Employing Proposition~\ref{chainrulecombine} and Lemma~\ref{nonemptylinear} ensures that $w \in -K_\Gamma(x,v-\nabla f(x))$, and that there exists a multiplier $y \in \Lambda(x,v)$ giving us
$$
\langle z - \nabla^2 f(x) w, -w\rangle - \langle w, \nabla^2 \langle y, g\rangle (x)(-w)\rangle\leq 0
$$
for all $x \in U$, $w \in \R^n$, $(x,v)\in \gph\partial\varphi$, and $z \in \breve{\partial}^2 \varphi(x,v)(w)$. Combining the latter with \eqref{sufficient2nd} yields
$$
\langle z, w\rangle \geq \langle \nabla^2 f(x)w, w\rangle + \langle \nabla^2 \langle y, g\rangle (x) w, w\rangle \ge 0
$$
when $z \in \breve{\partial}^2\varphi(x,v)(w)$, $(x,v)\in\gph\partial\varphi \cap (U\times V)$, and $w \in \R^n$. Thus \eqref{PSDsufcomposite} and \eqref{PSDsufcompositeII} are verified under the fulfillment of \eqref{sufficient2nd} and \eqref{sufficient2ndII}, respectively. The last statement of the theorem follows from Proposition~\ref{equitiltstr}, which therefore completes the proof. 
\end{proof}

To conclude this section, we discuss the relationships of the obtained results with other second-order subdifferential sufficient optimality conditions for local minimizers of \eqref{composite}.
 
\begin{Remark}\label{rem:2ndsuf} \rm Quite recently, some second-order sufficient conditions for local minimizers of problem \eqref{composite} with $\Theta= 0_{\R^m}$ have been derived in \cite{anyenxu23} via the regular second-order subdifferential from Definition~\ref{2ndsub}(iii). The main assumptions in \cite{anyenxu23} are the surjectivity of the Jacobian $\nabla g(\ox)$ and the positive-semidefiniteness of  $\widehat{\partial}^2\mathcal{L}(x)$ on the whole space, where $\mathcal{L}$ is the Lagrange function in  \eqref{composite}. Note that these assumptions imply the {\em local convexity} of the Lagrange function $\mathcal{L}$. Meanwhile, the set $\Theta$ in our consideration is a general polyhedral set, the assumptions in Theorem~\ref{suff2ndconstraint} do not imply the local convexity of the Lagrange function while being essentially weaker than the positive-semidefiniteness of  $\widehat{\partial}^2\mathcal{L}(x)$ on the whole space. Moreover, our constraint qualification is MSCQ, which is much weaker than the surjectivity of the Jacobian $\nabla g(\ox)$. 
\end{Remark}\vspace*{-0.2in}

\section{Applications to Nonlinear Programming}\label{sec:appnonl}\vspace*{-0.05in}

The main goal of this section is to establish {\em second-order necessary and sufficient optimality conditions} for local and strong local minimizers of {\em nonlinear programs} (NLPs) formulated as follows:
\begin{equation}\label{NLPproblem}
\min\quad f(x)\;\text{ subject to }\;
\begin{cases}
\varphi_i(x)\le 0&\text{for }\;i=1,\ldots,s,\\
\varphi_i(x)=0&\text{for }\;i=s+1,\ldots,m,
\end{cases}
\end{equation}
where $f$ is a function of class $\mathcal{C}^{1,1}$, and  where $\varphi_i$ for $i=1,\ldots,m$ are $\mathcal{C}^2$-smooth functions around the reference points. The optimality conditions for NLPs obtained below will be derived as specifications of those for the general constrained optimization problems in Sections~\ref{sec:optiforconstrained} and \ref{sec:2ndsuf} while being expressed in term of the initial data of \eqref{NLPproblem} under the corresponding version of the fairly mild MSCQ. 

Problem \eqref{NLPproblem} can be obviously written in the form of \eqref{composite} as 
$$
\min \quad f (x)  \; \text{ subject to }\; g(x) \in \Theta:=\R_{-}^{s}\times \{0_{m-s}\}
$$
\color{black}
with $g(x):=(\varphi_1(x),\ldots,\ph_m(x))$ and the set of the NLP feasible solutions is $\Gamma:=\big\{x \in \R^n\;\big|\; g(x) \in \Theta\big\}$. Define the corresponding {\em Lagrangian function} $L:\R^n\times\R^m\to\R$ by 
\begin{equation}\label{Lagrangefunction}
L(x,y):= f(x)+\langle y,g(x)\rangle\;\text{ for all }\;x\in\R^n,\;y\in\R^m
\end{equation}
together with the index collections
\begin{equation}\label{index}
I(x):=\big\{i\in\{1,\ldots,s\}\;\big|\;\ph_i(x)=0\big\} \;\text{ and }\; I_+(x,y):=\big\{i\in I(x)\;\big|\;y_i>0\big\}.
\end{equation}
Further, for each $(x,y)\in\R^n\times\R^m$ consider the cone 
\begin{equation}\label{coneS}
S(x,y):=\left\{w\in\R^n\;\bigg|\; \begin{cases}
\langle\nabla\ph_i(x),w\rangle=0&\text{if }\;i\in\big\{s+1,\ldots,m\big\}\cup I_+(x,y),\\
\langle\nabla\ph_i(x),w\rangle\geq  0 & \text{if }\;i\in I(x)\setminus I_+(x,y)
\end{cases} \right\}.
\end{equation}
Next we recall for the reader's convenience the classical constraint qualification conditions in NLP. Given a feasible solution $\ox$ to \eqref{NLPproblem}, it is said that the {\em linear independence constraint qualification} (LICQ) holds at $\ox$ of \eqref{NLPproblem} if the vectors
\begin{equation}\label{linearICQ}
\nabla\ph_i(\ox)\;\mbox{ for }\;i\in I(\ox) \cup \{s+1,\ldots,m\}\;\text{ are linearly independent},
\end{equation} 
where $I(\ox)$ stands for the collection of active inequality constraint indices at $\ox$ taken from \eqref{index}. The {\em positive linear independence constraint qualification} (PLICQ) is satisfied at $\ox$ if 
\begin{equation}\label{linearPICQ}
 {\bigg[\disp\sum_{i\in I(\ox)\cup\{s+1,\ldots,m\}}\al_i\nabla\ph_i(\ox)=0,\;\al_i\ge 0,\;i \in I(\ox)  \bigg] \Longrightarrow\big[\al_i=0\;\mbox{ if }\; i \in I(\ox)\cup \{s+1,\ldots,m\} \big].}
\end{equation} 
Note that both of these constraint qualifications are robust with respect to small perturbations of $\ox$. In fact, \eqref{linearPICQ} is a dual form of the {\em Mangasarian-Fromovitz constraint qualification} (MFCQ) at $\ox$.
It follows from \cite[Lemma~7.2]{kmp22convex} that PLICQ and LICQ  are equivalent to FOCQ \eqref{1stqualifyconcompos} and SOCQ \eqref{SOCQ}, respectively, and thus our major MSCQ \eqref{MSCQ} is weaker than PLICQ. For the NLP model \eqref{NLPproblem}, we can equivalently describe MSCQ \eqref{MSCQ} via the existence of a neighborhood $U$ of $\ox$ and $\kappa>0$ such that
\begin{equation}\label{MSCQNLP}
\dist(x;\Gamma) \leq \kappa \left(\sum_{i=1}^s \max\{\varphi_i(x),0\} +\sum_{i=s+1}^m |\varphi_i(x)|  \right)\;\text{ for all }\; x \in U.  
\end{equation}

To derive second-order necessary and sufficient conditions for \eqref{NLPproblem} in terms of the given data, explicit formulas for the set of multipliers and  the critical cone for NLPs are needed. Here they are.

\begin{Proposition}\label{calNLPex} Let $\Theta=\R_{-}^{s}\times \{0_{m-s}\}$. Then  we have
\begin{equation}\label{normalconeKs}
N_\Theta(z)= F_1(z_1)\times\ldots\times  F_m(z_m)\;{\text{ for all }\;z\in \Theta,}
\end{equation} 
\begin{equation}\label{tangentconeKs}
T_\Theta(z)= G_1(z_1)\times \ldots \times G_m(z_m) \; \text{ for all }\; z \in \Theta,
\end{equation}
where each multifunction $F_i\colon\R\tto\R$ and $G_i \colon \R \tto \R$ is defined by
\begin{equation*}
F_i(t):=\begin{cases}
[0,\infty)&\text{if}\quad t=0,\;1\le i\le s,\\
\{0\} &\text{if}\quad t<0,\;1\le i\leq s,\\
\R &\text{if}\quad t =0,\; s+1\le i\le m,\\
\emp &\text{otherwise}.
\end{cases} 
\end{equation*}
\begin{equation*}
G_i(t):=\begin{cases}
(-\infty,0]&\text{if}\quad t=0,\;1\le i\le s,\\
\R &\text{if}\quad t<0,\;1\le i\leq s,\\
\{0\} &\text{if}\quad t =0,\; s+1\le i\le m,\\
\emp &\text{otherwise}.
\end{cases} 
\end{equation*}
Consequently, the set of multipliers \eqref{Lagrangemultiplier} for problem \eqref{NLPproblem} is represented by
\begin{equation}\label{Lmulti2}
\Lambda(x,v)=\big\{y\in\R_+^s\times\R^{m-s}\;\big|\;v=\nabla_x L(x,y),\;y_i\ph_i(x)=0,\;i=1,\ldots,m\big\}
\end{equation} 
whenever $(x,v)\in\R^n\times\R^n$. If furthermore  MSCQ holds at some $\ox \in \Gamma =\{x\in \R^n\mid g(x)\in \Theta\}$, then for any $x \in \Gamma$ sufficiently close to $\ox$, any $v \in \R^n$, and any $y \in \Lambda(x,v)$, we have 
$$
S(x,y) =- K_\Gamma(x,v-\nabla f(x)),
$$
where the cone $S(x,y)$ is defined by \eqref{coneS}. 
\end{Proposition}
\begin{proof} It follows from \cite[Proposition~6.41]{Rockafellar98} and the obvious equality $\Theta=\Theta_1\times\ldots\times\Theta_m$ with
$$
\Theta_i:=\begin{cases}
(-\infty,0]&\text{if }\quad i=1,\ldots,s,\\
\{0\}&\text{if }\quad i=s+1,\ldots,m
\end{cases}
$$
that we have the normal and tangent cone representations
$$
N_\Theta(z)=N_{\Theta_1}(z_1)\times \ldots \times N_{\Theta_m}(z_m) \quad\text{for all }\;z=(z_1,\ldots,z_m)\in\Theta,
$$
$$
T_\Theta(z)=T_{\Theta_1}(z_1)\times \ldots \times T_{\Theta_m}(z_m) \quad\text{for all }\;z=(z_1,\ldots,z_m)\in\Theta.
$$
Since the sets $\Theta_i$ are convex for each $i\in\{1,\ldots,m\}$, it is easy to see that 
$$  
N_{\Theta_i}(t)=F_i(t) \; \text{and } \; T_{\Theta_i}(t) = G_i(t)\;\text{ whenever }\;t\in\R,\;i =1,\ldots,m,
$$
which verifies \eqref{normalconeKs} and \eqref{tangentconeKs}. Consequently, we get formula \eqref{Lmulti2} for the set of multipliers. To proceed further, suppose that MSCQ holds at $\ox \in \Gamma$ and deduce from 
\cite[Lemma~2.5]{chieuhien} that
\begin{align*}
T_\Gamma (x)& =  \big\{w \in \R^n\;\big|\; \nabla g(x) w \in T_\Theta(g(x))\big\}\\
&=\left\{w\in\R^n\;\bigg|\; \begin{cases}
\langle\nabla\ph_i(x),w\rangle=0&\text{if }\;i\in\big\{s+1,\ldots,m\big\},\\
\langle\nabla\ph_i(x),w\rangle\leq  0 & \text{if }\;i\in I(x) 
\end{cases} \right\}\;\mbox{ and}
\end{align*}
\begin{align*}
\big(v- \nabla f(x)\big)^\perp &=\big\{w\in\R^n\;\big|\; \langle w, v - \nabla f(x) \rangle =0\big\}\\
& =\big\{w\in\R^n\;\big|\;\langle w, \nabla_x L(x,y)-\nabla f(x) \rangle = 0\big\} \\
&= \left\{w \in \R^n\bigg|\; \left\langle \sum_{i=1}^m y_i \nabla \varphi_i(x), w \right\rangle = 0 \right\}\\
&= \left\{w \in \R^n\bigg|\;\sum_{i=1}^m y_i \left\langle    \nabla \varphi_i(x), w \right\rangle = 0 \right\}
\end{align*}
whenever $x \in \Gamma$ sufficiently close to $\ox$, and where $y \in \Lambda(x,v)$ as $v\in\R^n$. Therefore, for such vectors $x$, $v$, and $y$ we obtain the  following representations  
\begin{align*}
 K_\Gamma(x,v-\nabla f(x)) &  = T_\Gamma(x) \cap (v-\nabla f(x))^\perp \\
 &=  \left\{w\in\R^n\;\bigg|\; \begin{cases}
\langle\nabla\ph_i(x),w\rangle=0 \quad \text{if }\;i\in\big\{s+1,\ldots,m\big\},\\
\langle\nabla\ph_i(x),w\rangle\leq  0 \quad  \text{if }\;i\in I(x),\\
\sum_{i=1}^m y_i \left\langle    \nabla \varphi_i(x), w \right\rangle = 0
\end{cases} \right\}\\
&= \left\{w\in\R^n\;\bigg|\; \begin{cases}
\langle\nabla\ph_i(x),w\rangle=0 \quad \text{if }\;i\in\big\{s+1,\ldots,m\big\},\\
\langle\nabla\ph_i(x),w\rangle\leq  0 \quad  \text{if }\;i\in I(x),\\
\sum_{i\in I(x)} y_i \left\langle    \nabla \varphi_i(x), w \right\rangle + \sum_{i\notin I(x)} y_i \left\langle    \nabla \varphi_i(x), w \right\rangle= 0
\end{cases} \right\}\\
&= \left\{w\in\R^n\;\bigg|\; \begin{cases}
\langle\nabla\ph_i(x),w\rangle=0 \quad \text{if }\;i\in\big\{s+1,\ldots,m\big\},\\
\langle\nabla\ph_i(x),w\rangle\leq  0 \quad  \text{if }\;i\in I(x),\\
\sum_{i\in I(x)} y_i \left\langle    \nabla \varphi_i(x), w \right\rangle = 0
\end{cases} \right\}\\
&= \left\{w\in\R^n\;\bigg|\; \begin{cases}
\langle\nabla\ph_i(x),w\rangle=0 \quad \text{if }\;i\in\big\{s+1,\ldots,m\big\},\\
\langle\nabla\ph_i(x),w\rangle\leq  0 \quad  \text{if }\;i\in I(x),\\
 y_i \left\langle\nabla \varphi_i(x), w \right\rangle = 0 \quad  \text{if } i \in I(x)
\end{cases} \right\} \\
&= \left\{w\in\R^n\;\bigg|\; \begin{cases}
\langle\nabla\ph_i(x),w\rangle=0&\text{if }\;i\in\big\{s+1,\ldots,m\big\}\cup I_+(x,y),\\
\langle\nabla\ph_i(x),w\rangle\leq  0 & \text{if }\;i\in I(x)\setminus I_+(x,y) 
\end{cases} \right\}
\end{align*}
and therefore complete the proof of the proposition. 
\end{proof}

Having the above calculations in hand, we arrive at the {\em pointbased second-order necessary optimality conditions} for local and strong local minimizers of NLPs.

\begin{Theorem}\label{necforNLP} Consider the optimization problem \eqref{NLPproblem}, and let $\ox$ be a local minimizer of \eqref{NLPproblem}. If MSCQ \eqref{MSCQNLP} holds at $\ox$, then $\ox$ satisfies the first-order optimality condition
\begin{equation}\label{1stKKT}
\nabla_x L(\ox,\oy)=0\;\text{ and }\;\langle\oy,g(\ox)\rangle=0\;\text{ for some }\;\bar{y}\in\R_+^s\times\R^{m-s}.
\end{equation}
If in addition $\oy$ in \eqref{1stKKT} is unique $($in particular, when LICQ \eqref{linearICQ} is satisfied$)$, then 
\begin{equation}\label{necescompcondiNLP}
 \langle z, w\rangle +  \langle w, \nabla^2 \langle \oy, g\rangle (\ox)w\rangle \geq 0
\end{equation}
for all $w \in S(\ox,\oy)$  
and $z \in \breve{\partial}^2 f(\ox)(w)$.  If $\ox$ is a strong local minimizer of \eqref{NLPproblem} with modulus $\sigma>0$, then we have the condition
\begin{equation}\label{necescompcondiNLPII}
 \langle z, w\rangle +  \langle w, \nabla^2 \langle \oy, g\rangle (\ox)w\rangle \geq \sigma\|w\|^2
\end{equation}
for all    $w \in S(\ox,\oy)$  
and $z \in \breve{\partial}^2 f(\ox)(w)$. 
\end{Theorem}
\begin{proof} We immediately obtain \eqref{necescompcondiNLP} and \eqref{necescompcondiNLPII} from Theorem \ref{necforconstrained} and Proposition \ref{calNLPex}.
\end{proof} 

The next theorem provides new {\em neighborhood second-order sufficient optimality conditions} for local and strong local minimizers in the classical NLP setting with ${\cal C}^2$-smooth data under the weakest MSCQ. We actually get more by establishing {\em variational} and {\em strong variational convexity} of the extended-real-valued function associated with \eqref{NLPproblem}.

\begin{Theorem}\label{sufforNLP} Consider the optimization problem \eqref{NLPproblem} in which all the functions $f$ and $\ph_i$ are $\mathcal{C}^2$-smooth around a feasible solution $\ox$ satisfying the first-order optimality condition  \eqref{1stKKT}.
Supposing that MSCQ holds at $\ox$, we have the following assertions:

{\bf(i)} If there exist neighborhoods $U$ of $\ox$ and $V$ of $0$ such that 
\begin{equation}\label{suffcompcondiNLP}
 \langle \nabla^2 f(x)w, w\rangle +  \langle w, \nabla^2 \langle y, g\rangle (x)w\rangle \geq 0 \; \text{ whenever }\;x \in U,\;v \in V   \; \text{ and }\;w \in S(x,y),
\end{equation}
where $y \in \R^s_+ \times \R^{m-s}$ is a solution to the equations $v=\nabla_x L(x,y)$ and $\langle y, g(x)\rangle=0$, then the function $f+ \delta_\Gamma$ is variationally convex at $\ox$ for $0$. Consequently, $\ox$ is a local minimizer of \eqref{NLPproblem}.

{\bf(ii)} If there exist $\sigma >0$, neighborhoods $U$ of $\ox$ and $V$ of $0$ such that 
\begin{equation}\label{suffcompcondiNLPII}
\langle \nabla^2 f(x)w, w\rangle +  \langle w, \nabla^2 \langle y, g\rangle (x)w\rangle \geq \sigma\|w\|^2 \; \text{ whenever }\; x \in U,\;v \in V,  \; \text{ and }\;w \in S(x,y),
\end{equation}
where $y \in \R^s_+ \times \R^{m-s}$ is a solution to the equations $v=\nabla_x L(x,y)$ and $\langle y, g(x)\rangle=0$, then the function $f+ \delta_\Gamma$ is strongly variationally convex at $\ox$ for $0$ with modulus $\sigma$. Consequently, $\ox$ is a strong local $($in fact tilt-stable$)$ minimizer of \eqref{NLPproblem} with modulus $\sigma^{-1}$.
\end{Theorem}
\begin{proof}
It follows immediately from Theorem~\ref{suff2ndconstraint} and Proposition~\ref{calNLPex}.
\end{proof}

The next example illustrates our newly obtained second-order sufficient optimality condition in which the results in 
Theorem~\ref{sufforNLP} can be applied, while the sufficient conditions from \cite[Theorem~4.1]{anyenxu23} and
\cite[Corollary~7.4]{kmp22convex}  cannot be used.

\begin{Example}\label{exa:suf}
\rm Consider the optimization problem  taken from Example~\ref{nlpex}. This problem is the NLP class being equivalently written in the form
\begin{equation}\label{NLPexam}
\min\quad f(x_1,x_2):=x_1^2-x_2^2\;\text{ subject to }\;
\begin{cases}
\varphi_1(x_1,x_2):=-x_1+x_2 &\le 0, \\
\varphi_2(x_1,x_2):=-x_1-x_2 &\leq 0,\\ 
\varphi_3(x_1,x_2):=x_2 &=0. 
\end{cases} 
\end{equation}
The Lagrangian function $L:\R^2 \times \R^3\to \R$ is given here by
$$
L(x_1,x_2,y_1,y_2,y_3)= x_1^2-x_2^2 + y_1(-x_1+x_2)+ y_2(-x_1-x_2)+y_3 x_2
$$
for all $(x_1,x_2,y_1,y_2,y_3) \in \R^5$. It is easy to check that the reference point $\ox:=(0,0)$ satisfies the first-order optimality condition \eqref{1stKKT}. 
Since LICQ fails at $\ox:=(0,0)$, we cannot apply the sufficient conditions from \cite[Theorem~4.1]{anyenxu23} and \cite[Corollary~7.4]{kmp22convex}, while Example~\ref{nlpex} shows that MSCQ is satisfied at any point of the plane. Considering the corresponding mapping
$$
g(x_1,x_2):= \big(\varphi_1(x_1,x_2), \varphi_2(x_1,x_2), \varphi_3(x_1,x_2)\big)\;\text{ on }\;\R^2, 
$$
we see that whenever $x=(x_1,x_2) \in \R^2$ and $y=(y_1,y_2,y_3) \in \R^3$, the inclusion $w:=(w_1,w_2) \in S(x,y)$ implies that $w_2 =0$. Therefore, we arrive at the condition
$$
\langle \nabla^2 f(x)w, w\rangle +\langle w, \nabla^2 \langle y, g\rangle (x) w\rangle = 2w_1^2 -2w_2^2 = 2w_1^2 \geq 0,
$$
which gives us \eqref{suffcompcondiNLP}. Hence $\ox$ is a local minimizer of \eqref{NLPexam}. 
\end{Example}

Finally in this section, we present the following consequence of Theorem~\ref{sufforNLP} providing the {\em pointbased} condition for {\em strong variational convexity} and {\em sufficiency} of strong (even tilt-stable) local minimizers in \eqref{NLPproblem} under the fulfillment of PLICQ/MFCQ at the reference point. Note that this sufficient condition for tilt-stable minimizers was previously obtained for NLPs with only inequality constraints by using a completely different technique not based on second-order subdifferentials.

\begin{Corollary} Consider NLP \eqref{NLPproblem} in which all the functions are  $\mathcal{C}^2$-smooth around a feasible solution $\ox$ satisfying the first-order optimality condition  \eqref{1stKKT}.
Assume that PLICQ holds at $\ox$ and that the strong second-order sufficient condition $($SSOSC) is satisfied, i.e., for all $\oy$  from \eqref{1stKKT} we have 
\begin{equation}\label{pointPDNLP}
\langle \nabla^2 f(\ox)w, w\rangle +  \langle w, \nabla^2 \langle \oy, g\rangle (\ox)w\rangle >0 \; \text{ if }\;\langle \nabla \varphi_i(\ox),w\rangle =0 \; \text{for }\;i \in\{s+1,\ldots,m\}\cup I_{+}(\ox,\oy). 
\end{equation}
Then the function $f+ \delta_\Gamma$ is strongly variationally convex at $\ox$ for $0$. 
Consequently, $\ox$ is a tilt-stable $($hence strong$)$ local minimizer of \eqref{NLPproblem} with some modulus $\kappa>0$. 
\end{Corollary}
\begin{proof} Assuming that \eqref{pointPDNLP} is satisfied, we need to show that there exist $\sigma>0$, neighborhoods $U$ of $\ox$ and $V$ of $0$ such that \eqref{suffcompcondiNLPII} holds. Suppose the contrary and find sequences $x_k\to \ox$, $v_k\to 0$, $\sigma_k \to 0$ as $k \to \infty$ together with $w_k \in S(x_k,y_k)$ and $y_k \in \Lambda(x_k,v_k)$ such that 
\begin{equation}\label{pc1}
\langle \nabla^2 f(x_k)w_k, w_k\rangle + \langle w_k, \nabla^2 \langle y_k, g\rangle (x_k)w_k\rangle < \sigma_k \|w_k\|^2\;\mbox{ for all }\;k\in\N, 
\end{equation}
which implies that $w_k \ne 0$. Defining $\hat{w}_k:=w_k/\|w_k\|$ gives us $\hat{w}_k \in S(x_k,y_k)$ since $S(x_k,y_k)$ is a cone, and $\|\hat{w}_k\|=1$ for all $k \in \N$. These allow us to rewrite \eqref{pc1} in the form 
\begin{equation}\label{pc2}
\langle \nabla^2 f(x_k)\hat{w}_k, \hat{w}_k\rangle + \langle \hat{w}_k, \nabla^2 \langle y_k, g\rangle (x_k)\hat{w}_k\rangle < \sigma_k, \quad k \in \N. 
\end{equation}
We have without loss of generality that $\hat{w}_k \to w$ as $k \to \infty$ with some  $w \ne 0$. It follows from 
\cite[Lemma~7.2]{kmp22convex} that PLICQ is equivalent to FOCQ \eqref{1stqualifyconcompos}. Thus the set-valued mapping $\Lambda$ from \eqref{Lagrangemultiplier} is locally bounded around $(\ox,0)$ by Proposition~\ref{local-closed}, which yields the boundedness of the sequence $\{y_k\}$. Again without loss of generality, suppose that $y_k \to \oy$ as $k \to \infty$ and conclude that $\oy \in \Lambda(\ox,0)$ due to the closedness of the graphical set $\gph \Lambda$. 

To check further the conditions   
\begin{equation}\label{iintwoset}
\langle \nabla \varphi_i(\ox),w\rangle =0 \;\text{ for }\;i \in\{s+1,\ldots,m\}\cup I_{+}(\ox,\oy), 
\end{equation} 
fix $i \in \{s+1,\ldots,m\} \cup I_{+}(\ox,\oy)$ and consider the following two cases: 

\textbf{Case~1:} $i \in \{s+1,\ldots,m\}$. Since $\hat{w}_k \in S(x_k,y_k)$ for all $k\in \N$, we have 
$$
\langle \nabla \varphi_i(x_k), \hat{w}_k\rangle =0\; \text{ for all }\; k \in \N. 
$$
Passing there to the limit as $k \to \infty$ brings us to $\langle \nabla \varphi_i(\ox), w\rangle =0$  for all $i \in \{s+1,\ldots,m\}$.

{\bf Case~2:} $i\in  I_+(\ox,\oy)$. In this case we have $\varphi_i(\ox)=0$ and $\oy_i >0$. Since $y_k \to \oy$, it follows that $(y_k)_i >0$ for sufficiently large $k\in \N$. Combining the latter with $(y_k)_i\varphi_i(x_k) =0$ tells us that $\varphi_i(x_k) =0$ and hence $i \in I_{+} (x_k,y_k)$ for all large $k$. This yields 
$$
\langle \nabla \varphi_i(x_k), \hat{w}_k\rangle =0\;\text{ for sufficiently large }\; k \in \N
$$
as $\hat{w}_k \in S(x_k,y_k)$. Letting $k \to \infty$ in the above equalities brings us to $\langle \nabla \varphi_i(\ox), w\rangle =0$, which justifies \eqref{iintwoset}. Passing now to the limit as $k \to \infty$ in \eqref{pc2}, we get 
$$
\langle \nabla^2 f(\ox)w, w\rangle + \langle w, \nabla^2 \langle \oy, g\rangle (\ox)w\rangle \leq 0, 
$$
which being combined with \eqref{iintwoset} contradicts  \eqref{pointPDNLP}. Thus \eqref{suffcompcondiNLPII} holds for some $\sigma>0$. Employing Theorem~\ref{sufforNLP} confirms that $f+\delta_\Gamma$ is strongly variationally convex at $\ox$ for $0$, which implies by Proposition~\ref{equitiltstr} that $\ox$ is a tilt-stable local minimizer and thus completes the proof of the corollary. 
\end{proof}\vspace*{-0.2in}

\section{Concluding Remarks}\label{sec:conclusion}\vspace*{-0.05in}

This paper establishes new second-order necessary and sufficient optimality conditions for both unconstrained and constrained optimization problems in finite-dimensional spaces that utilize second-order subdifferentials of extended-real-valued prox-regular functions. The obtained results resolve, in particular, some challenging open questions formulated in the recent paper \cite[Section~5]{anyenxu23}. 

Our future research will address developing the second-order subdifferential approach of this paper to deriving second-order necessary and sufficient optimality conditions for nonpolyhedral problems of conic programming and extended nonlinear programming expressed entirely in terms of the given data under the weakest constraint qualifications. We also aim at applying the obtained optimality conditions and associated results on variational convexity and variational sufficiency to the design and justification of advanced second-order numerical algorithms.

\section*{Acknowledgments}
The authors are grateful to both anonymous referees for various helpful remarks and suggestions that allowed us to improve the original presentation.

\end{document}